\documentclass[11pt,reqno]{amsart}
\theoremstyle{definition}

\theoremstyle{remark}

\numberwithin{equation}{section}
\usepackage{amsmath, amssymb, amsthm, mathrsfs%, makeidx, fancyheadings,
}
\usepackage{hyperref}
\newtheorem{athm}{Theorem}

\usepackage{url}
\theoremstyle{plain}

\newtheorem{theorem}{Theorem}[section]
\newtheorem{proposition}[theorem]{Proposition}
\newtheorem{corollary}[theorem]{Corollary}
\newtheorem{lemma}[theorem]{Lemma}

\theoremstyle{definition}

\newtheorem{example}{Example}
\newtheorem{definition}[theorem]{Definition}

\newtheorem{remark}[theorem]{Remark}

\DeclareMathOperator{\esssup}{ess \ sup}

\DeclareMathOperator{\diam}{diam}
\DeclareMathOperator{\Ker}{Ker}

\input{tcilatex}

\begin{document}
	
	%-------------------------------------------------------------------------
	% editorial commands: to be inserted by the editorial office
	%
	%\firstpage{1} \volume{228} \Copyrightyear{2004} \DOI{003-0001}
	%\seriesextraline{This is the Concrete Title of this Book\br H.E. R and S.T.C. W, Eds.}
	%\firstpage{1}

	\title[Lipschitz regularity of the i.m. of Random Dynamical Systems]
	{Lipschitz regularity of the invariant measure of Random Dynamical Systems}

	%----------Author 1
	\author{Davi Lima}
	
	\address{%
		Av Lourival Melo Mota s/n,
		Alagoas-
		Brazil}
	
	\email{davi,.santos@im.ufal.br}
	
	\thanks{This work was partially supported by CNPq and FAPEAL}
	%----------Author 2
	\author{Rafael Lucena}
	\address{%
		Av Lourival Melo Mota s/n,
		Alagoas-Brazil}
	\email{rafael.lucena@im.ufal.br}
	
	
	\date{April 19, 2024}
	\dedicatory{To our great friend Helton Ferreira Albuquerque Medeiros (in memorian)}
	%%% ----------------------------------------------------------------------
	
	\begin{abstract}
		In this article we derive a regularity result for the disintegration of the invariant measure associated to a class of Random Dynamical Systems - RDS. The results of this work are obtained by  constructing a suitable anisotropic normed space defined by the Wasserstein-Kantorovich-like metric and understanding the dynamics of the associated transfer operator in a neighborhood of its fixed point. Precisely, we employ functional analytic techniques to demonstrate a spectral gap for its action on suitable spaces of signed measures. We apply this analysis to prove an exponential decay of correlation statement for Lipschitz observables and statistical properties of the RDS.
	\end{abstract}
	
	%%% ----------------------------------------------------------------------
	\maketitle
	%%% ----------------------------------------------------------------------
	%\tableofcontents
	
	%\thanks{2010 \emph{Mathematics Subject Classification}: Primary 37D20; Secondary 37C20.}

	\tableofcontents
	\section{Introduction}
	
	%Among relevant properties concerning the invariant measures, it is important to know how much regular it is. This sort of problem can be seen under "different" point of views. For instance, if the measure is absolutely continuous w.r. to a reference measure, you can ask about the regularity using the regularity of its Radon-Nykodim derivative (see \cite{L2}). If otherwise there is no reference measure information about the regularity of the disintegration w.r. a regular measurable partition is useful to many applications (see \cite{GLu}, \cite{RRR} and \cite{RRRSTAB}). 
	
	Regarding relevant properties of invariant measures, assessing their degree of regularity is useful. This type of problem can be approached from different perspectives. For instance, if the invariant measure of the system is absolutely continuous with respect to a reference measure, its regularity can be evaluated based on the regularity of its Radon-Nikodym derivative (see \cite{L2}). On the other hand, in the absence of a reference measure, understanding the regularity of the disintegration with respect to a measurable partition can be valuable for various applications (see \cite{GLu}, \cite{RRR}, and \cite{RRRSTAB}).
	
	%In this work we prove that the disintegration of the invariant measure of a Random Dynamical Systems - RDS is Lipschitz, provided it is a contracting fiber map. The regularity of disintegrations has been studied in recent years and some works that can be cited are \cite{GLu}, \cite{RRR}, \cite{RRRSTAB}, \cite{GP} and \cite{BM}. In \cite{GLu} the authors obtained a bounded variation regularity for the disintegration of the invariant measure of Lorenz-like maps, with respect to a certain topology. In that work, this property was used to prove statistical stability of the invariant measure under a class of deterministic perturbations. In \cite{RRR} the authors proved Holder regularity for the disintegration for a class of piecewise partially hyperbolic maps, semi-conjugated to a non uniformly expanding systems. They used this fact to obtain exponential decay of correlations for Holder observables and others limit theorems. In \cite{RRRSTAB}, the Holder regularity was used to prove a general statistical stability results for the equilibrium states of partially hyperbolic maps. 
	
	In this work, we demonstrate that the disintegration of the invariant measure of a Random Dynamical System (RDS) is Lipschitz continuous, provided the fiber map is contracting. Some recent research has focused on the regularity of disintegrations and works in the field include \cite{GLu}, \cite{RRR}, \cite{RRRSTAB}, \cite{GP}, and \cite{BM}. In \cite{GLu}, the authors established bounded variation regularity for the disintegration of the invariant measure of Lorenz-like maps with respect to a specific topology, which they used to demonstrate the statistical stability of the invariant measure under certain deterministic perturbations. Similarly, in \cite{RRR}, the authors proved Hölder regularity for the disintegration in a class of piecewise partially hyperbolic maps semi-conjugated to non-uniformly expanding systems. This finding enabled them to obtain exponential decay of correlations for Hölder observables and other limit theorems. In \cite{RRRSTAB}, Hölder regularity was leveraged to prove statistical stability results for the equilibrium states of partially hyperbolic maps.
	
	%Here we investigate the action of the transfer operator associated for a class of random dynamical systems. We define suitable anisotropic spaces by lifting a spectral gap of a subshift of finite type to the transfer operator of the skew-product map. We follow the general ideas introduced by \cite{GLu} and \cite{RRR}, generalizing the techniques and applying them to the system presented here. Roughly speaking, we get a fixed point of the transfer operator oh the skew-product in a certain anisotropic space by lifting a spectral gap which is known on the basis.
	We examine the action of the transfer operator associated with a class of random dynamical systems. We establish appropriate anisotropic spaces by extending a spectral gap from a subshift of finite type to the transfer operator of the skew-product map. Following the general approach outlined in \cite{GLu} and \cite{RRR}, we generalize the techniques and apply them to the system under study. In essence, we achieve a fixed point of the transfer operator of the skew-product within a specific anisotropic space by lifting to the RDS a known spectral gap from the base system.

	This type of study is often carried out using the Ionescu-Tulcea and Marinescu Theorem by constructing a pair of suitable function spaces—a stronger space and an auxiliary weaker space—such that the action of the Perron-Frobenius operator on the stronger space exhibits a spectral gap (see \cite{Ba}, \cite{L2}, \cite{BG}, \cite{G}, and \cite{V} for introductory texts). Commonly, these approaches achieve regularity through compact inclusion arguments. However, our approach differs; we obtain regularity by examining the behavior of the transfer operator in the neighbourhood of the constructed fixed point. Specifically, we achieve uniqueness and regularity as a result of convergence to the equilibrium and the presence of a spectral gap.
	
	We use these facts to obtain exponential decay of correlations over Lipschitz observables. This approach gives stronger convergence results than \cite{GLu} and open the possibility to obtain stronger statistical statements. In the examples, as a simplest case of the main system of this work, some results on hyperbolic iterated functions systems (IFS) are also provided, where we conclude some properties for its invariant measure and obtain limit theorems. 
	
	%In recent years the technique of finding good anisotropic norms is applied to several dynamical systems. (see e.g.  \cite{B},\cite{D} for recent papers containing a survey of the topic). In this spirit we constructed anisotropic spaces using the topology defined by the Wasserstein-Kantorovich like metric applied a RDS, and get the conditional measures of the disintegration of the invariant measure varies in a Lipschitz manner with respect to the variation of the basis (identified here with the space of the atoms of the partition which provides the disintegration), which is a space of sequences. 
	We apply these results to achieve exponential decay of correlations for Lipschitz observables. Moreover, this approach yields stronger convergence results than those in \cite{GLu}(altough for diiferente class of maps) and opens up opportunities for making more robust statistical statements. In the examples, as a particular case, we provide results on hyperbolic iterated function systems (IFS), leading to conclusions about the properties of its invariant measure and the derivation of limit theorems.
	
	%In this paper we present a new approach and obtain spectral gap and convergence to the equilibrium results for random dynamical systems. 

	\medskip
	
	\noindent\textbf{Organization of the article.} The paper is structured as follows:
	
	\begin{itemize}
		\item Section \ref{sec1}: we introduce the kind of systems we consider in the paper. Essentially, it is system of the type $F(x,y)=(\sigma(\underline{x}), G(\underline{x},y))$, where $\sigma$ is the left shift. Here and until section \ref{kfdjfkjdfeddere}, where more regularity is required, we do not ask any kind of regularity on $G$ in the horizontal direction (for the functions $\underline{x} \longmapsto G(\underline{x},y)$, $y$ fixed);
		
		\item Section \ref{sec:spaces}: we introduce the functional spaces used in the paper and
		discussed in the last paragraph;
		
		\item Section \ref{eryet}: we show the basic properties of the transfer operator of $F$
		when applied to these spaces. In particular we see that there is a useful
		\textquotedblleft Perron-Frobenius\textquotedblright-like formula (see
		Proposition \ref{niceformulaab});
		
		\item Section \ref{yfjdhf}: we discuss the basic properties of the iteration of the
		transfer operator on the spaces we consider. In particular, we prove a \emph{%
			Lasota-Yorke inequality and a convergence to equilibrium statement} (see
		Propositions \ref{LYinfty} and \ref{5.8});
		
		\item Section \ref{wqweer}: we use the convergence to equilibrium and the
		Lasota-Yorke inequalities to prove the \emph{spectral gap} for the transfer
		operator associated to the system restricted to a suitable strong space (see
		Theorem \ref{spgap});
		
		\item Section \ref{kfdjfkjdfeddere}: we consider a similar system with some more
		regularity on the family of functions $\{G(\cdot,y)\}_{y \in K}$, $\underline{x} \longmapsto G(\underline{x}, y)$: %there exists a partition (cylinders) $\mathcal{P} = P_1, \cdots, P_{\deg (\sigma)}$, 
		we ask that the restriction of the function $\underline{x} \longmapsto G(\underline{x}, y)$ is $k_{y}$-Lipschitz, where the family $\{k_{y}\}_{y\in K}$ is uniformly bounded. For this sort of system, we prove a stronger regularity result for the iteration of probability measures (see Theorem \ref{hdjfhsdfjsd} and Remark \ref{kjedhkfjhksjdf}) and show that the $F$-invariant invariant measure has a Lipschitz disintegration along the stable fibers (see Theorem \ref{regu});
		
		\item Section \ref{decayy}: we use the Lipschitz regularity of the invariant measure established in section \ref{kfdjfkjdfeddere}, to prove that the abstract set of functions on which the system has decay of correlations (see Proposition \ref{kkkskdjd}) contains all Lipschitz functions (see Proposition \ref{ksjdhsdd}) and we finalize the section with the main decay of correlations statement, Theorem \ref{skdsdas}.
		
	\end{itemize}

\section{Settings\label{sec1}}

In this subsection we give the settings of our results. First, We fix the notation
$$L_{X,Y}(\varphi):=\sup_{x\neq y}\dfrac{d_Y(\varphi(x),\varphi(y))}{d_X(x,y)},$$
if $\varphi:(X,d_X)\rightarrow (Y,d_Y)$ is a Lipstchitz map from the metric space $(X,d_X)$ to the metric space $(Y,d_Y).$ If $Y=\mathbb{R}$ we denote $L_{X,\mathbb{R}}(f)$ just by $L_X(f)$. This is the standard notation except in two cases, the first one when $Y=\mathbb{R}$ and $(X,d_X)=(\Sigma^{+}_A,d_{\theta})$ which is defined below and the second case in section \ref{kfdjfkjdfeddere}. In the case below, for subshifts of finite type we just denote $|\varphi|_{\theta}:=L_{\Sigma^{+}_A,\mathbb{R}}(\varphi).$ 

\subsubsection{Sub-shifts of finite type\label{sec2}}
Let $A$ be an aperiodic matrix and $\Sigma^+_A$ the subshift of finite type associated to $A$, i.e. 
$$\Sigma^+_A=\{\underline{x}=\{x_i\}_{i\in \mathbb{N}}; A_{x_ix_{i+1}}=1 \ \ \forall i\}, $$where $x_i \in \{1, \cdots, N\}$ for all $i$. 
On $\Sigma^+_A$ we consider the metric $$d_\theta (\underline{x},\underline{y}) = \sum _{i=0} ^{\infty}{\theta ^i (1-\delta (x_i,y_i))},$$where $\theta \in (0,1)$ and $\delta$ is defined by $\delta(x,y)=1$ if $x=y$ and $\delta(x,y)=0$ if $x\neq y$. 
We consider a Markov measure associated to $(\Sigma^+_A, \sigma)$ and denote it by $m$. Here %For a function $\varphi:\Sigma^+_A \longrightarrow \mathbb{R}$ we denote its Lipschitz constant by $|\varphi|_\theta$, i.e. 
%$$|\varphi|_\theta: = \sup_{\underline{x}\neq \underline{y}\in \Sigma^+_A}\left\{ \dfrac{|\varphi(\underline{x})-\varphi(\underline{y})|}       {d_\theta(\underline{x},\underline{y})}\right\}.$$  \marginpar{Incluir que notação universal para m\'etrica Lipschitz}
Let $\mathcal{F}_{\theta}(\Sigma^+_A)$ be the real vector space of the Lipschitz functions on $\Sigma^+_A$, i,e.
$$\mathcal{F}_{\theta}(\Sigma^+_A):=\{\varphi:\Sigma^+_A \longrightarrow \mathbb{R}: |\varphi|_\theta <\infty \}$$ 
endowed with the norm $||\cdot ||_\theta$, defined by $||\varphi||_{\theta} = ||\varphi||_\infty + |\varphi|_{\theta})$. In this case, we know that the shift map $\sigma:\Sigma^+_A\rightarrow \Sigma^+_A$, defined by $\sigma(\underline{x})_i=x_{i+1}$ for all $i \geq 0$, is such that its Perron-Frobenius operator\footnote{The unique operator $P_{\sigma}:L_1(m) \longrightarrow L_1(m) $ such that for all $\varphi \in L_1(m) $ and for all $\psi \in L_{\infty}(m) $ it holds $\displaystyle{\int {\varphi \cdot \psi \circ \sigma}dm = \int {P_{\sigma}(\varphi) \cdot \psi}dm}$.}, $P_{\sigma}:\mathcal{F}_{\theta}(\Sigma^+_A) \longrightarrow \mathcal{F}_{\theta}(\Sigma^+_A)$, has spectral gap.
It means that $P_{\sigma}:\mathcal{F}_{\theta}(\Sigma^+_A) \longrightarrow \mathcal{F}_{\theta}(\Sigma^+_A)$ can be written (see \cite {Ba} or apply Theorem 1 of \cite{RE}) as
$$\func{P}_{\sigma}=\Pi_{\sigma}+\mathcal{N}_\sigma$$
where the spectral radius of $\mathcal{N_{\sigma}}$ is smaller than 1, that is, $\rho(\mathcal{N}_\sigma)<1$, and $\Pi _\sigma$ is a projection ($\Pi _\sigma ^2= \Pi _\sigma$). So, there are $D>0$ and $r\in (0,1)$ 	such that for every $\varphi\in \ker(\Pi _\sigma)$ it holds
$$||\func{P}_{\sigma}^n\varphi||_{\theta}\le Dr^n||\varphi||_{\theta}, \ \ \forall n \geq 0.$$

\subsubsection{Contracting Fiber Maps\label{sec22}}
Let $(K,d)$ be a compact metric space equipped with its Borel $\sigma$-algebra, which is generated by the open sets. We define the map $$F:\Sigma^+_A\times K\rightarrow \Sigma^+_A\times K$$ by 
\begin{equation}\label{cccccc}
	F(\underline{x},z)=(\sigma(\underline{x}),G(\underline{x},z)),
\end{equation}where $G: \Sigma^+_A\times K \longrightarrow K$ is a measurable function, and the space $\Sigma^+_A\times K$ is equipped with the product $\sigma$-algebra. 

For the sake of simplicity, in certain parts of the text, we denote a fiber simply as $\gamma$ instead of $\gamma_{\underline{x}}$. However, in the proofs of the results presented in this paper, we make a clear distinction between these notations. Additionally, to avoid the use of multiplicative constants, we assume $\diam(K)=1$. 
Throughout this paper, we denote by $\pi_1$ and $\pi_2$ the projections on $\Sigma^+_A$ and $K$, respectively. 

Let $(X, \mathcal{X})$ be a measurable space, and let $\mu$ be a measure on $X$. If $f: X \longrightarrow \mathbb{R}$ is a function in $L_\mu ^1$, we define the signed measure $f\mu$ on $(X, \mathcal{X})$ by $$f\mu(E) := \int _E f d\mu,$$ for all $E \in \mathcal{X}$.

%=N_{1}\times N_{2}$, where $N_{1}$ and $N_{2}$ are compact and finite dimensional Riemannian manifolds such that $\diam (N_2)=1$, where $\diam (N_2)$ denotes the diameter of $N_2$ with respect to its Riemannian metric, $d_2$. This is not restrictive but will avoid some multiplicative constants. Denote by $m_{1}$ and $m_{2}$ the Lebesgue measures on $N_1$ and $N_2$ respectively, generated by their corresponding Riemannian volumes, normalized so that $m_{1}(N_{1})=m_{2}(N_{2})=1$ and $m=m_{1}\times m_{2}$. Consider a map
%F:\left(\Sigma, m \right)
%\longrightarrow \left(\Sigma, m \right)$, 
%\begin{equation}
%F(x,y)=(T(x),G(x,y)),  \label{1eq}
%\end{equation}%
%where $T:N_{1}\longrightarrow N_{1}$ and $G:\Sigma \longrightarrow N_{2}$
%are measurable maps. Suppose that these maps satisfy the following conditions

\subsubsection{Properties of $G$}
In what follows, we present the two main hypotheses regarding the coordinate function $G$. We emphasize that (G2) is a necessary hypothesis only for Theorems \ref{regu} and \ref{skdsdas}.

\begin{description}
	\item[G1] Consider the $F$-invariant lamination 
	\begin{equation}
		\mathcal{F}^{s}:=\{\gamma_{\underline{x}}\}_{\underline{x}\in \Sigma_A^+}, \ \ \textnormal{where} \ \gamma_{\underline{x}}: = \{ \underline{x}\} \times K.
		\label{fol}
	\end{equation} 
	We suppose that $\mathcal{F}^{s}$ is contracted: there
	exists $0<\alpha <1$ such that for $m$-almost all $\underline{x}\in \Sigma_A^+$ it holds%
	\begin{equation}
		d(G(\underline{x},y_{1}),G(\underline{x},y_2))\leq \alpha d(y_{1},y_{2}), \quad \forall
		y_{1},y_{2}\in K.  \label{contracting1}
	\end{equation}%\marginpar{Davi, coloquei G2 aqui. Acredito que não há problema nisso, pois os teoremas deixam claro quando G2 é de fato utilizada. Não removi G2 da página 19. Acho que devemos remover se decidirmos deixar ela aqui.}
	\item[G2]For all $y \in K$ there exists $k_y \geq 0$ such that  
	$$d(G(\underline{x}^1, y),G(\underline{x}^2, y))\leq k_yd_{\theta}(\underline{x}^1,\underline{x}^2).$$Moreover
	\begin{equation}\label{hhfksdjfksdfsdfsd}
		H:=\sup_{y\in K} k_y<\infty.
	\end{equation}
\end{description}
%We know that ( cf \cite{ZZ}, for example ) under $\textbf{G1}$ the map $F$ has an invariant measure $\mu_0$ such that $\pi^*\mu_0=m$. 

The example below illustrates how our formalism can be naturally applied within the context of a standard contracting IFS.

\begin{example}\label{kjfhjsfg}
	Let $ I_1, \cdots, I_d$ be closed and disjoint intervals in $\mathbb{R}$ and let $I$ be the convex hull of them. Consider $g: I_1 \cup ... \cup I_d \longrightarrow I$ such that $|g'| \geq \lambda >1$ and $g(I_i)=I$ diffeomorphically for all $i=1, ..., d $. Set $\varphi_i:=(g|_{I_i})^{-1}:I \longrightarrow I_i$, $A=(a_{ij})_{i,j=1}^n$ where $a_{ij}=1$ for all $i,j\in \{1,2,...,d\}$ and $G:\Sigma_A^+ \times I \longrightarrow I$, by $G(\underline{x}, y) := \varphi_{ x_0}(y)$, if $\underline{x}=(x_0,x_1,...)$. Define $F: \Sigma_A^+ \times I \longrightarrow \Sigma_A^+ \times I $, by  $F(\underline{x},y)= (\sigma (\underline{x}), G(\underline{x}, y))$.
	We note that by the assumptions on $g$ and the Mean Value Theorem it follows
	$$d(G(\underline{x},y_1),G(\underline{x},y_2))=d(\varphi_{x_0}(y_1),\varphi_{x_0}(y_2))\le \lambda^{-1}d(y_1,y_2)$$ 
	this shows that \textbf{G1} is satisfied, where $K=I$ and $\alpha=\lambda^{-1}.$

	Now we verify the hypothesis \textbf{G2}. If $\underline{x}^1=(x^1_i)_{i\in \mathbb{Z}^+},\underline{x}^2=(x^2_i)_{i\in \mathbb{Z}^+}$ belong to a same cylinder $[0;j]$ then
	$$d(G(\underline{x}^1,y),G(\underline{x}^2,y)=d(\varphi_{x^1_0}(y),\varphi_{x^2_0}(y))=0$$
	while otherwise 
	$$d(G(\underline{x}^1,y),G(\underline{x}^2,y)=d(\varphi_{x^1_0}(y),\varphi_{x^2_0}(y))\le \diam(I)=\dfrac{\theta \diam(I)}{\theta}.$$
	but $d_{\theta}(\underline{x},\underline{y})=\sum_{i=0}^{\infty} \theta^i(1-\delta(x^1_i,y^2_i))\ge 1>\theta$, because $x_0\neq y_0.$ Therefore,
	$$d(G(\underline{x}^1,y),G(\underline{x}^2,y)\le R d_{\theta}(\underline{x}^1,\underline{x}^2)$$
	where $R=\dfrac{\diam(I)}{\theta}<\infty$. Since the constant $R$ found here does not depend on $y$, the previous computation shows that \textbf{G2} is also satisfied and the number $H$ as in the equation (\ref{hhfksdjfksdfsdfsd}) is such that $H\le R$.
\end{example}
In \cite{Hc}, Hutchinson, J.E. proved the existence of an invariant measure: a measure $\mu$ on $K$ which satisfies the relation 
\begin{equation}\label{nvbbjdjfdf}
	\mu = \sum_{1}^{N}p_i\varphi_{i}^{\ast}\mu,
\end{equation}where $(p_1, p_2, \cdots, p_N)$ is a probability vector, $\sum_{1}^{N}p_i=1$ and $\varphi_{i}^{\ast}$ denotes the push-foward of $\mu$ by $\varphi_i$, that is, $\varphi_i^{\ast}\mu(E)=\mu(\varphi_i^{-1}(E))$. Moreover, if $m$ is the Bernoulli measure defined by $p_1, \cdots, p_N$, then the product $m \times \mu$ is invariant by the skew product $F: \Sigma_A^+ \times K \longrightarrow \Sigma_A^+ \times K $, defined by the above example. This model will serve as a guide for the application discussed in Section \ref{bggvd} (see Example \ref{dereroidf}).

\subsection{Statements of the Main Results}\label{sotmr}

In this section, we present the main results of this article, which are derived from analysing the action of the transfer operator $\func{F}^{\ast }$ associated with $F$. The operator $\func{F}^{\ast }$  is defined for a given signed measure $\mu$ and a measurable set $E \subset \Sigma_A^+ \times K$ by the expression: $$\func{F}^{\ast}\mu (E):=\mu(F^{-1}(E)).$$ Our approach involves analysing the action of $\func{F}^{\ast}$ on appropriate normed spaces of signed measures, denoted by $(\mathcal{L}^{\infty}, ||\cdot||{\infty})$ and $(S^{\infty}, ||\cdot||_{S^{\infty}})$. These spaces are constructed using a variant of the Wasserstein-Kantorovich metric (see Definition \ref{wasserstein}). Informally, $\mathcal{L}^\infty$ (see Definition \ref{L}) represents the space of signed measures $\mu$ defined on $\Sigma_A^+ \times K$, such that the family of its disintegration along $\mathcal{F}^s$ (see equation (\ref{fol})) is uniformly bounded with respect to the Wasserstein-Kantorovich metric. On the other hand, $S^\infty$ (see Definition \ref{iuytietrsdf}) is the space of signed measures $\mu$ on $\Sigma_A^+ \times K$ satisfying the following conditions:

\begin{enumerate}
	\item The projection of $\mu$ onto the first coordinate, $\pi_1^{\ast} \mu$, where $\pi_1(x, y) = x$ for all $x \in \Sigma_A^+$ and $y \in K$, is absolutely continuous with respect to $m$;
	\item The density of $\pi_1^{\ast} \mu$ with respect to $m$, denoted $\frac{d \pi_1^{\ast} \mu}{d m}$, belongs to the space $\mathcal{F}_{\theta}(\Sigma^+_A)$.
\end{enumerate}Here, $\pi_1^*$ denotes the push-forward operator. This framework enables a rigorous examination of the properties of $\func{F}^{\ast}$, its effect on the measures in $S^\infty$, and its implications for the dynamics of $F$.

\begin{athm}\label{probun}
	Suppose that $F$ satisfies (G1). Then, there exists a unique $F$-invariant probability measure $\mu _0 \in S^{\infty}$ which satisfies $||\mu_0||_\infty =1$ and $||\mu_0||_{S^\infty}=2$.
\end{athm}

The second theorem demonstrates that $\func{F}^{\ast }: S^{\infty} \longrightarrow S^{\infty}$ satisfies the \textit{spectral gap} property. This result has far-reaching implications for the dynamics, as it leads to several limit theorems. For example, we obtain an exponential rate of convergence for the limit $$\lim_n {C_n(f,g)}=0,$$where $$C_n(f,g):=\left| \int{(g \circ F^n )  f}d\mu_0 - \int{g  }d\mu_0 \int{f  }d\mu_0 \right|,$$with $g: \Sigma_A^+ \times K \longrightarrow \mathbb{R} $ as a Lipschitz function and $f \in \Theta _{\mu _0}$. The set $\Theta _{\mu _0}$ is defined by 
\begin{equation}\label{teta}
	\Theta _{\mu _0}:= \{ f: \Sigma_A^+ \times K \longrightarrow \mathbb{R}; f\mu_0 \in S^\infty\},
\end{equation}

\begin{athm}[Spectral gap]\label{spgap}
	If $F$ satisfies (G1), then the operator $\func{F}^{\ast
	}:S^{\infty}\longrightarrow S^{\infty}$ can be written as 
	\begin{equation*}
		\func{F}^{\ast }=\func{P}+\func{N},
	\end{equation*}%
	where
	
	\begin{enumerate}
		\item[a)] $\func{P}$ is a projection i.e. $\func{P} ^2 = \func{P}$ and $\dim
		Im (\func{P})=1$;
		
		\item[b)] there are $0<\xi <1$ and $M>0$ such that $\forall \mu \in S^\infty$ $$||\func{N}^{n}(\mu )||_{S^{\infty}}\leq M\xi ^{n} ||\mu||_{S^{\infty}} ;$$ 
		
		\item[c)] $\func{P}\func{N}=\func{N}\func{P}=0$.
	\end{enumerate}
\end{athm}

The third theorem provides an estimate for the Lipschitz constant (see equation (\ref{Lips2}) in Definition \ref{Lips3}) associated with the disintegration of the unique $F$-invariant measure $\mu_0 \in S^{\infty}$. This type of result has numerous applications, and similar estimates for other systems can be found in \cite{GLu}, \cite{BM} and \cite{RRR}. For instance, in \cite{GLu}, the regularity of the disintegration is used to demonstrate the stability of the $F$-invariant measure under a particular type of \textit{ad-hoc} perturbation. Here, we leverage this result to show that the abstract set $\Theta _{\mu _0}$, which is defined by equation (\ref{teta}), includes Lipschitz functions.

In the following result, we denote by $\mathcal{L}_\theta ^{+}$ the space of positive measures on $\Sigma_A^+ \times K$ for which the Lipschitz constant of their disintegration along $\mathcal{F}^s$, denoted by $|\mu|_\theta$, is finite (see Definitions \ref{Lips3} and \ref{erfcscvdsd}).

\begin{athm}\label{regu}
	Suppose that $F$ satisfies (G1) and (G2). Let $\mu_0$ be the unique $F$-invariant probability measure in $S^{\infty}$. Then $\mu _{0}\in \mathcal{L}_\theta ^{+}$ and 
	\begin{equation*}
		|\mu _{0}|_\theta \leq \dfrac{C_1}{1-\theta},
	\end{equation*}where $C_1>0$ is a constant. %was defined in Theorem \ref{hdjfhsdfjsd}, by $C_1 = \max \{ H\theta +\theta N |g|_\theta , 2\}$.
\end{athm}

The following result is a consequence of all previous theorems. It shows that the system $F$ has exponential decay of correlations ($\lim {C_n(f,h)}=0$ exponentially fast) for Lipschitz observables $f,h$. 

%Remember that the Lipschitz constant $L(g)$ of a Lipschitz function $g:X\longrightarrow \mathbb{R}$, where $(X,d)$ is a compact metric space, is defined as follows %$L(g)$ be its best  i.e. 
%\begin{equation*}
%\displaystyle{L(g)=\sup_{x \neq y}\left\{ \dfrac{|g(x)-g(y)|}{d(x,y)}%
	%	\right\} }.
%\end{equation*}
%We use $|\varphi|_{\theta}$ instead $L(\varphi)$, if we consider $(X,d)=(\Sigma^{+}_A,d_{\theta}).$ 
Denote by $\widehat{\mathcal{F}}(\Sigma^+_A\times K)$ the set of real Lipschitz functions, $h:\Sigma^+_A \times K \longrightarrow \mathbb{R}$, with respect to the metric $d_\theta + d$. The Lipschitz norm of $h\in \widehat{\mathcal{F}}(\Sigma^+_A\times K)$ is defined by $||h||_{Lip}=||h||_{\infty}+L_{\Sigma^{+}_A\times K}(h)$.

%And for such a function we denote by $\Lu _\theta (f)$ its Lipschitz constant. 

\begin{athm}\label{skdsdas}
	Suppose that $F$ satisfies (G1) and (G2). For all Lipschitz functions $h, f \in \widehat{\mathcal{F}}(\Sigma^+_A \times K)$, it holds $$\left| \int{(h \circ F^n )  f}d\mu_0 - \int{h  }d\mu_0 \int{f  }d\mu_0 \right| \leq M \xi ^{n}  ||f \mu _0||_{S^{\infty}} ||h||_{Lip}  \ \ \forall n \geq 1,$$where $\xi$ and $M$ are from Theorem \ref{spgap} .
	%and $|g|_{\theta} := |g|_\infty + \Lu _\theta (g)$. 
\end{athm}

\section{Preliminares \label{sec:spaces}}

%\subsection{$L^{\infty}$-like spaces.}

%Through this section we construct some function spaces which are suitable for
%the systems we consider. The idea is to consider spaces of signed measures, with suitable norms constructed by disintegrating measures along
%the stable foliation. Thus a signed measure will be seen as a family of measures on
%each leaf. As an example, a measure on the square will be seen as a one
%parameter family (a path) of measures on the interval (a stable leaf). In
%the vertical direction (on the leaves), we will consider a norm which is the dual of the Lipschitz norm. In the \textquotedblleft horizontal\textquotedblright direction we will consider
%essentially the $L^{\infty}$ norm.

\subsection*{Rokhlin's Disintegration Theorem} Now we give a brief introduction about disintegration of measures.

Consider a probability space $(\Sigma,\mathcal{B}, \mu)$ and a partition $\Gamma$ of $\Sigma$ by measurable 
sets $\gamma \in \mathcal{B}$. Denote by $\pi : \Sigma \longrightarrow \Gamma$ the projection that associates
to each point $u \in \Sigma$ the element $\gamma _u$ of $\Gamma$ which contains $u$, i.e. $\pi(u) = \gamma _u$. Let $\widehat{\mathcal{B}}$ 
be the $\sigma$-algebra of $\Gamma$ provided by $\pi$. Precisely, a subset $\mathcal{Q} \subset \Gamma$ is measurable if, and only if, $ \pi^{-1}(\mathcal{Q}) \in \mathcal{B}$. We define the \textit{quotient} measure $\hat{\mu}$ on $\Gamma$ by $\hat{\mu}(\mathcal{Q})= \mu(\pi ^{-1}(\mathcal{Q}))$.

The proof of the following theorem can be found in \cite{Kva}, Theorem 5.1.11.

\begin{theorem}(Rokhlin's Disintegration Theorem) Suppose that $\Sigma$ is a complete and separable metric space, $\Gamma$ is a measurable partition\footnote{We say that a partition $\Gamma$ is measurable if there exists a full measure set $M_0 \subset \Sigma$ s.t. restricted to $M_0$, $\Gamma = \bigvee _{n=1} ^{\infty} \Gamma_n$, for some increasing sequence $\Gamma _1\prec \Gamma_2 \prec \cdots \prec \Gamma_n \prec \cdots $ of countable partitions of $\Sigma$. Furthermore, $\Gamma_i \prec \Gamma_{i+1} $ means that each element of $\Gamma_{i+1}$ is a subset of some element of $\Gamma_i$. } of $\Sigma$ and $\mu$ is a probability on $\Sigma$. Then, $\mu$ admits a disintegration relatively to $\Gamma$, i.e. a family $\{\mu _\gamma\}_{\gamma \in \Gamma}$ of probabilities on $\Sigma$ and a quotient measure $\hat{\mu}= \pi ^* \mu$ such that:
	
	\begin{enumerate}
		\item [(a)] $\mu _\gamma (\gamma)=1$ for $\hat{\mu} $-a.e. $\gamma \in \Gamma$;
		\item [(b)] for all measurable sets $E \subset \Sigma$, the function $\Gamma \longrightarrow \mathbb{R}$, defined by $\gamma \longmapsto  \mu _\gamma (E)$ is measurable; 
		\item [(c)] for all measurable sets $E \subset \Sigma$, it holds $\mu (E) = \int {\mu _\gamma(E)}d\hat{\mu}(\gamma)$.
	\end{enumerate}
	\label{rok}
\end{theorem}The proof of the following lemma can be found in  \cite{Kva}, Proposition 5.1.7.
\begin{lemma}
	Suppose the $\sigma$-algebra $\mathcal{B}$, of $\Sigma$, has a countable generator. If $(\{ \mu_\gamma \}_{\gamma \in \Gamma},  \hat{\mu}  )$ and $(\{ \mu'_{\gamma } \}_{\gamma \in \Gamma}, \hat{\mu} )$ are disintegrations of the measure $\mu$ relatively to $\Gamma$, then $\mu_\gamma = \mu'_\gamma$, $\hat{\mu}$-almost every $\gamma \in \Gamma$. 
	\label{kv}
\end{lemma}

\subsection{The $\mathcal{L}^{\infty}$ and $S^{\infty}$ spaces}

Fix $\Sigma=\Sigma_A^+\times K$ and let $\mathcal{SB}(\Sigma )$ be the space of Borel signed measures on $\Sigma$. Given $\mu \in \mathcal{SB}(\Sigma )$ denote by $\mu ^{+}$ and $\mu ^{-}$
the positive and the negative parts of its Jordan decomposition, $\mu =\mu ^{+}-\mu ^{-}$ (see Remark {\ref{ghtyhh}). Let $\mathcal{AB}$ be the set \begin{equation}
		\mathcal{AB}=\{\mu \in \mathcal{SB}(\Sigma ):\pi_1^{\ast }\mu ^{+}\ll m\
		\ \mathnormal{and}\ \ \pi_1^{\ast }\mu ^{-}\ll m\},  \label{thespace1}
	\end{equation}
	where $m$ is the Markov measure and $\pi_1$ is defined by $\pi_1(\underline{x},y)=\underline{x}$, where $\underline{x} \in \Sigma^+_A$ and $y \in K$.
	%Let us denote by $\mu_{x}$ the absolutely continuous marginal measure $\mu_{x}=\pi _{x}^{\ast }\mu = \phi _{x}m _1$. 
	%Let us consider a finite positive measure $\mu \in \mathcal{AB}$ on the space  $\Sigma$  foliated by the contracting leaves $\mathcal{F}^{s}=\{\gamma_x\}_{x\in N_1 }$ such that  $\gamma_x={\pi _{x}}^{-1}(x) $. 
	
	Given a {\em probability measure} $\mu \in \mathcal{AB}$, we briefly outline how Theorem \ref{rok} is applied in our context. For $u=(\underline{x},y) \in \Sigma$, consider $\gamma_u=\{\underline{x}\}\times K$. In particular, if $\pi_1(u)=\pi_1(\tilde{u})$, then $\gamma_{u}=\gamma_{\tilde{u}}$. Let $\Gamma=\{\gamma_u:=\{\underline{x}\}\times K; \underline{x}\in \Sigma^+_A\}$.  With this setup, the map $\pi(u)=\gamma_u$ serves as a quotient map, $\pi:\Sigma\rightarrow \Gamma.$ Observe that $\Gamma$ is naturally identified with $\mathcal{F}^s=\{\gamma_{\underline{x}}=\{\underline{x}\}\times K; \underline{x}\in \Sigma^+_A\}$ and $\pi$ is naturally identified with the projection $\pi_1$. Moreover, since $\dfrac{d\hat{\mu}}{dm}=\phi_1$, it follows that $\hat{\mu}=\phi_1 m$. Consequently, we can naturally identify $\hat{\mu}$ with the absolutely continuous measure $\phi_1 m$.

	Below we write $\pi _{2,\gamma}(\underline{x},y)=y$ if $\gamma=\gamma_{\underline{x}}.$ Using $\pi_{2,\gamma}$ we define a family of signed measures on $K$.
	%be the restriction $\pi
	%_2|_{\gamma }$, where $\pi _2:\Sigma \longrightarrow K$ is the
	%projection on $K$
	%and $\gamma \in \mathcal{F}^s$.
	
	\begin{definition} Given a positive measure $\mu \in \mathcal{AB}$ and its disintegration along the stable leaves $\mathcal{F}^{s}$, $\left( \{\mu _{\gamma} \}_{\gamma},\hat{\mu}=\phi _1m \right) $ we define the \textbf{restriction of $\mu $ on $\gamma $} as the positive measure $\mu |_{\gamma }$ on $K$ (not on the leaf $\gamma $) defined for all measurable set $B\subset K$ by 
		\begin{equation*}
			\mu |_{\gamma}(B):=\pi _{2, \gamma}^{\ast }(\phi_1(\underline{x})\mu _{\gamma}
			)(B), \ \mbox{where} \ \gamma=\gamma_{\underline{x}}.\label{restrictionmeasure}
		\end{equation*}
		%The above definition means precisely: 
		%\begin{equation*}
		%\mu |_{\gamma_{\underline{x}}}(E):=\pi _{2, \gamma_{\underline{x}}}^{\ast }(\phi_1(\underline{x} )\mu _{\gamma_{\underline{x}}}
		%)(E).
		%\end{equation*}
		
	\end{definition}
}

\begin{lemma}
	If $\mu \in \mathcal {AB}$, then for each measurable set $B \subset K $, the function $\tilde{c}: \Sigma _A ^+ \longrightarrow \mathbb{R}$, given by $\tilde{c}(\underline{x}) = \mu |_\gamma (B)$, is measurable, where $\gamma = \gamma _ {\underline{x}}$.
\end{lemma}
\begin{proof}For a given $\mu \in \mathcal {AB}$, by Theorem \ref{rok}, we have that the function $\gamma \mapsto \mu _\gamma (B)$ is measurable. Moreover, $\phi _1$ is measurable. Since, $$\tilde{c}(\underline{x}) = \mu |_\gamma (B) 
	= \phi _1(\underline{x}) \mu_\gamma (\pi_{2,\gamma}^{-1}(B) )
	=\phi _1(\underline{x}) \mu_\gamma (\pi_{2}^{-1}(B) \cap \gamma)
	=\phi _1(\underline{x}) \mu_\gamma (\pi_{2}^{-1}(B))$$ we have $\tilde{c}$ is a measurable function.	 
	
\end{proof}

For a given signed measure $\mu \in \mathcal{AB}$ and its Jordan decomposition $\mu
=\mu ^{+}-\mu ^{-}$, define the \textbf{restriction of $\mu $ on $\gamma $}
by%
\begin{equation}
	\mu |_{\gamma }:=\mu ^{+}|_{\gamma }-\mu ^{-}|_{\gamma }.
\end{equation}%

\begin{remark}\label{ghtyhh}
	As we prove in Proposition \ref{lasttttt}, the restriction $\mu |_\gamma$ does not depend on the decomposition. Precisely, if $\mu = \mu_1 - \mu_2$, where $\mu_1$ and $\mu _2$ are any positive measures, then $\mu|_\gamma := \mu_1|_\gamma - \mu_2|_\gamma=\mu ^{+}|_{\gamma }-\mu ^{-}|_{\gamma }$ $m$-a.e. $\underline{x} \in \Sigma^+_A$, where $\gamma = \gamma_{\underline{x}}$. 
\end{remark}

\begin{definition}
	Given two signed measures $\mu $ and $\nu $ on $X,$ we define a \textbf{%
		Wasserstein-Kantorovich Like} distance between $\mu $ and $\nu $ by%
	\begin{equation}
		W_{1}^{0}(\mu ,\nu )=\sup_{L_X(g)\leq 1,||g||_{\infty }\leq 1}\left\vert \int {%
			g}d\mu -\int {g}d\nu \right\vert .
	\end{equation}
	\label{wasserstein}
\end{definition}From now on, we denote%
\begin{equation}\label{WW}
	||\mu ||_{W}:=W_{1}^{0}(0,\mu ).
\end{equation}%
As a matter of fact, $||\cdot||_{W}$ defines a norm on the vector space of
signed measures defined on a compact metric space.
We remark that this norm is equivalent to the dual of the Lipschitz norm. 

Other applications of this metric to obtain limit theorems can be seen in \cite{GP}, \cite{RRR} and \cite {ben}. For instance, in \cite {ben} the author apply this metric to a more general case of shrinking fibers systems.  	

\begin{remark}
	Henceforth, we denote $\vert \vert \mu |_\gamma \vert \vert _W : = W_{1}^{0}(\mu ^+|_\gamma ,\mu ^-|_\gamma )$.
\end{remark}

%\begin{definition}
%Let $\mathcal{L}^{1}\subseteq \mathcal{AB}$ be defined as%
%\begin{equation}
%\mathcal{L}^{1}=\left\{ \mu \in \mathcal{AB}:\int _{N_1} {W_{1}^{0}(\mu
	%^{+}|_{\gamma },\mu ^{-}|_{\gamma })}dm _1(\gamma )<\infty \right\}
%\label{L1measurewithsign}
%\end{equation}and define a norm on it, $||\cdot ||_{1}:\mathcal{L}^{1}\longrightarrow \mathbb{R}$, by%
%\begin{equation}
%||\mu ||_{1}=\int _{N_1} {W_{1}^{0}(\mu ^{+}|_{\gamma },\mu ^{-}|_{\gamma })}%
%dm _1(\gamma ).  \label{l1normsm}
%\end{equation}%
%\label{l1likespace}
%\end{definition}

%Now, we define the following set of signed measures on $\Sigma $,%
%\begin{equation}
%S^{1}=\left\{ \mu \in \mathcal{L}^{1};\phi _{x}\in S_{\_}\right\}.
%\end{equation} Consider the function $||\cdot ||_{S^{1}}:S^{1}\longrightarrow \mathbb{R}$, defined by%
%\begin{equation}
%||\mu ||_{S^{1}}=|\phi _{x}|_{s}+||\mu ||_{1},
%\end{equation}%
%where we denote $\phi _{x}=\phi _{x}^{+}-\phi _{x}^{-}$  with $\phi _{x}^{\pm}$ being the marginals of  $\mu ^{\pm}$ as explained before. Moreover, $\phi _{x}$ is the marginal density of the disintegration of $\mu $ and we remark that $\phi _{x}^{+}$ is not necessarily equal to the positive part of $\phi _{x}$.

%in \cite{L}.

%\begin{proposition}
%$\left( \mathcal{L}^{1},||\cdot ||_{1}\right) $ and $\left( S^{1},||\cdot
%||_{S^{1}}\right) $ are normed vector spaces. \label{propnorm1}
%\end{proposition}

%\subsection{$L^{\infty }$ like spaces}

\begin{definition}\label{L}
	Let $\mathcal{L}^{\infty }\subseteq \mathcal{AB}(\Sigma )$ be defined as%
	\begin{equation}
		\mathcal{L}^{\infty }:=\left\{ \mu \in \mathcal{AB}:\esssup_{\underline{x}} \{||\mu |_{\gamma_{\underline{x}}} ||_W : \underline{x} \in \Sigma _A^+\}<\infty \right\},
	\end{equation}%
	where the essential supremum is taken over $\Sigma^+_A$ with respect to $m$. Define the function $||\cdot ||_{\infty }:\mathcal{L}^{\infty }\longrightarrow \mathbb{R}$
	by%
	\begin{equation}
		||\mu ||_{\infty }=\esssup \{W_{1}^{0}(\mu ^{+}|_{\gamma_{\underline{x}} },\mu ^{-}|_{\gamma_{\underline{x}}
		});\underline{x}\in \Sigma_A^+\}.
	\end{equation}
\end{definition}

Finally, we define the following set of signed measures on $\Sigma $%

\begin{definition}\label{iuytietrsdf}
	\begin{equation}\label{si}
		S^{\infty }=\left\{ \mu \in \mathcal{L}^{\infty };\phi_1\in
		\mathcal{F}_{\theta}(\Sigma^+_A)\right\}, 
	\end{equation}and the function, $||\cdot ||_{S^{\infty }}:S^{\infty }\longrightarrow \mathbb{R}$, defined by%
	\begin{equation}
		||\mu ||_{S^{\infty }}=||\phi_1||_{\theta}+||\mu ||_{\infty },
	\end{equation}where $\phi_1=\dfrac{d(\pi_1^*\mu)}{dm}$.
\end{definition}

It is straightforward to show that $\left( \mathcal{L}^{\infty },||\cdot ||_{\infty }\right) $ and $\left(
S^{\infty },||\cdot||_{S^{\infty }}\right) $ are normed vector spaces, see for instance \cite{L}.
Consider $\mathcal{SB}(K)$ with the Borel's sigma algebra generating with the Wasserstein-Kantorovich Like metric. We have that the map $\tilde{c}:\Sigma_A^+\rightarrow  \mathcal{SB}(K), \underline{x}\mapsto \mu|_{\gamma_{\underline{x}}}$ is a measurable function.	In fact, note that by Theorem \ref{rok} the map $\underline{x} \mapsto \mu_{\gamma_{\underline{x}}}$ from $\Sigma_A^+$ to $\mathcal{SB}(\Sigma)$, $\phi_1: \Sigma_A^+ \longrightarrow \mathbb{R}$ and $\pi^*_{2,\gamma}:\mathcal{SB}(\Sigma)\rightarrow \mathcal{SB}(K)$ are measurable functions.

%\newpage

\section{Transfer operator associated to $F$}\label{eryet}

Henceforth, we consider the transfer operator $\func{F}^{\ast }$ associated with 
$F$, that is, $\func{F}^{\ast }$ is given by
\begin{equation*}
	\lbrack \func{F}^{\ast }\mu ](E)=\mu (F^{-1}(E)),
\end{equation*}
for each signed measure $\mu \in \mathcal{SB}(\Sigma) $ and for each measurable set $E\subset \Sigma $.

For the proof of the following lemma, we define $P_i:=[0;i]=\{\underline{x}=(x_m)_{m\in \mathbb{N}}; x_0=i\}$ and denote $\sigma_i:= \sigma |_{P_i}$ for all $i \in \{1, \cdots, N\}$. As above, if $\mu \in \mathcal{AB}$, we write $(\{\mu _{\gamma }\}_\gamma,\phi_1)$ to describe its disintegration along $\mathcal{F}^s$, where $\phi_1=\dfrac{d(\pi_1^*\mu)}{dm}$.
\begin{lemma}
	\label{transformula}If  $\mu \in \mathcal{AB}$ is a probability measure and $\phi_1:=\dfrac{d(\pi_1^*\mu)}{dm}$, then $F^*\mu\in \mathcal{AB}$ and
	\begin{equation}
		\dfrac{d\pi_1^*({F}^{\ast }\mu)}{dm}=\func{P}_{\sigma}(\phi _1). \label{1}
	\end{equation}
	Moreover, if for $\underline{x}=(x_n)_{n\in \mathbb{Z}^+}\in \Sigma^{+}_A$ we denote $\gamma=\gamma_{\underline{x}}$ and $\gamma _i := \gamma _{\sigma _i ^{-1}(\underline{x})}$, it holds
	\begin{equation}
		(\func{F}^{\ast }\mu )_{\gamma }=\nu_{\gamma}:=\frac{1}{\func{P}_{\sigma}(\phi _1)(\underline{x})}
		\sum_{i=1}^N{\frac{\phi _1}{J_{m,\sigma_i}}\circ \sigma_{i}^{-1}(\underline{x} )\cdot
			\chi _{\sigma_{i}(P_{i})}(\underline{x} )\cdot \func{F}^{\ast }\mu _{\gamma_i}
		}  \label{2}
	\end{equation}
	when $\func{P}_{\sigma}(\phi _1)(\underline{x})\neq 0$. Otherwise, if $\func{P}
	_{\sigma}(\phi _1)(\underline{x})=0$, then $\nu _{\gamma }$ is the Lebesgue measure on $\gamma$, where the sum is understood to be zero if the transition $ix_0$ is not allowed, that is, $A_{i,x_0}=0$, here $\underline{x}=(x_{\ell})_{\ell\ge 0}$. Moreover, throughout the article $\chi _{E}$ denotes the characteristic function of the set $E$.
\end{lemma}

\begin{proof}
	
	By the uniqueness of the disintegration (see Lemma \ref{kv}) to prove Lemma \ref{transformula}, is enough to prove the following equation 
	\begin{equation}
		\func{F}^{\ast }\mu (E) = \int_{\Sigma^+_A}{\nu_{\gamma }(E\cap \gamma )} \func{P}_{\sigma}(\phi _1)(\underline{x} ) dm(\underline{x}), 
	\end{equation}for a measurable set $E \subset \Sigma $.  
	In order to do that, let us define the sets $B_1 =\left\lbrace \underline{x} \in \Sigma _A^+; \func{P}_{\sigma} (\phi_1) (\underline{x}) =0 \right\rbrace$ and $B_2 = B_1 ^{c}$. It follows from the definition of $P_{\sigma}$ that 
	
	\begin{equation}\label{ergdg}
		\phi _1 m(\sigma^{-1}(B_1))=0. 
	\end{equation}

	Below we write $S:=\int_{\Sigma^+_A}{\nu_{\gamma }(E\cap \gamma )} \func{P}_{\sigma}(\phi _1)(\underline{x} ) dm(\underline{x})$. Using the change of variables $\underline{x} = \sigma_i(\underline{z}_i)$ ($\underline{z}_i=i\underline{x} \in P_i$), equation (\ref{ergdg}), and the definition of $\nu _\gamma$  (see equation (\ref{2})), we have 
	\begin{eqnarray*}
		S&=&\int_{B_{2}}{\sum_{i=1}^N{\ {\frac{\phi _1}{J_{m,\sigma_i}%
					}\circ \sigma_{i}^{-1}(\underline{x} )\func{F}^{\ast }\mu _{\gamma _i}(E)\chi_{\sigma_{i}(P_{i})}(\underline{x})}}}%
		dm(\underline{x} )
		\\
		&=&\sum_{i=1}^N{\int_{\sigma_{i}(P_{i})\cap B_{2}}{\ {\frac{\phi _1}{J_{m,\sigma_i}%
					}\circ \sigma_{i}^{-1}(\underline{x})\func{F}^{\ast }\mu _{\gamma_i}(E)}}}%
		dm(\underline{x} ) \\
		&=&\sum_{i=1}^N{\int_{P_{i}\cap \sigma_{i}^{-1}(B_{2})}{\ {\phi _1(\underline{z}_i)\mu
					_{\gamma _ {\underline{z}_i}}(F^{-1}(E))}}}dm(\underline{z}_i) \\
		&=&{\int_{\sigma^{-1}(B_{2})}{\ {\phi _1(\underline{z}_i)\mu
					_{\gamma _ {\underline{z}_i}}(F^{-1}(E))}}}dm(\underline{z}_i)\\
		&=&\int_{\bigcup_{s=1}^{2}{\sigma^{-1}(B_{s})}}{\ {\ \mu _{\gamma _ {\underline{z}_i}}(F^{-1}(E))}}%
		d(\phi_1m)(\underline{z}_i) \\
		&=&\int_{\Sigma^+_A}{\ {\ \mu _{\gamma _ {\underline{z}_i}}(F^{-1}(E))}}{d(\phi _1m)}(\underline{z}_i) \\
		&=&\mu (F^{-1}(E)) \\
		&=&\func{F}^{\ast }\mu (E).
	\end{eqnarray*}And the proof is done.
\end{proof}

As said in Remark \ref{ghtyhh}, Proposition \ref{lasttttt} yields that the restriction $\mu|_{\gamma}$
does not depend on the decomposition. Thus, for each  $\mu \in \mathcal{L}^{\infty}$, since $\func{F}^*\mu$ can be decomposed as $\func{F}^*\mu = \func{F}^*(\mu^{+}) -\func{F}^*(\mu^-)$, we can apply the above Lemma to $\func{F}^*(\mu^+)$ and $\func{F}^*(\mu^-)$ to get the following

\begin{proposition}
	Let  $\gamma \in \mathcal{F}^{s}$ be a stable leaf. 
	Define the map $F_{\gamma}:K\longrightarrow K$ by $$F_{\gamma }=\pi _{2}\circ F|_{\gamma
	}\circ \pi _{2,\gamma}^{-1}.$$
	Then, for each $\mu \in \mathcal{L}^{\infty}$ and for $m$-almost every $\underline{x} \in \Sigma_A^+$ it holds 
	\begin{equation}
		(\func{F}^{\ast }\mu )|_{\gamma }=\sum_{i=1}^N{\dfrac{\func{F}%
				_{\gamma_i}^{\ast }\mu |_{\gamma_i }}{J_{m,\sigma_i}(i\underline{x} )}\chi _{\sigma_i(P _{i})}(\underline{x})},  \label{niceformulaa}
	\end{equation}where $\gamma=\gamma_{\underline{x}}, \ \ \underline{x}=(x_n)_{n\in \mathbb{Z}^+}$, $\gamma _i := \gamma _{\sigma _i ^{-1}(\underline{x})}$ and $i\underline{x} = \sigma _i ^{-1}(\underline{x})$, where the sum is understood to be zero if the transition $(ix_0)$ is not allowed.
	\label{niceformulaab}
\end{proposition}
\section{Basic properties of the norms and convergence to equilibrium}\label{yfjdhf}

In this section, we show important properties of the norms and their behaviour with respect to the transfer operator.
In particular, we show that the  $\mathcal{L}^{\infty}$ norm is weakly contracted by the transfer operator.
We prove a type of Lasota-Yorke inequality and an exponential convergence to equilibrium statement. All these properties will be used in next section to prove a spectral gap statement for the transfer operator.

\begin{proposition}[The weak norm is weakly contracted by $\func{F}^{\ast }$]\label{l1}
	
	Suppose that $F$ satisfies (G1). If $\mu \in \mathcal{L}^{\infty}$ then 
	\begin{equation}
		||\func{F}^{\ast }\mu ||_{\infty}\leq ||\mu ||_{\infty}.
	\end{equation}%
	\label{weakcontral11234}
\end{proposition}

In the proof of the proposition we will use the following lemma about the behavior of the $||\cdot ||_W$ norm (see equation (\ref{WW})) after a contraction. It says that a contraction cannot increase the  $||\cdot ||_W$ norm.

\begin{lemma}\label{niceformulaac}
	For every $\mu \in \mathcal{AB}$ and $m$-almost all stable leaf $\gamma \in \mathcal{F}^s$, it holds
	\begin{equation}
		||\func{F}_{\gamma }^{\ast }(\mu |_{\gamma })||_{W}\leq ||\mu |_{\gamma }||_{W},
	\end{equation} \label{weak1}
	where $F_{\gamma }:K\longrightarrow K$ is defined in Proposition \ref{niceformulaab}.
	Moreover, if $\mu $ is a probability measure on $K$, it holds
	\begin{equation}
		||\mu ||_{W}=1.
		\label{simples}
	\end{equation}
	In particular, since $F^*\mu$ is also a probability we have that $||\func{F}^{\ast }\mu ||_{W}=1$.
\end{lemma} 

\begin{proof}(of Lemma \ref{niceformulaac})
	Indeed, since $F_{\gamma }$ is an $\alpha$-contraction, if 
	$|g|_{\infty }\leq 1$ and $Lip(g)\leq 1$ the same holds for $g\circ
	F_{\gamma }$. Since
	\begin{eqnarray*}
		\left\vert \int {g~}d\func{F}_{\gamma }^{\ast }(\mu|_\gamma) \right\vert &=&\left\vert
		\int {g\circ F_{\gamma }~}d\mu |_\gamma \right\vert, \\ 
	\end{eqnarray*}%
	taking the supremum over $|g|_{\infty }\leq 1$ and $Lip(g)\leq 1$ we finish the proof of the inequality.

	In order to prove the equation (\ref{simples}), consider a probability measure $\mu$ on $K$ and a Lipschitz function $g:K \longrightarrow \mathbb{R}$, such that $||g||_\infty \leq 1$ and $L(g)\leq 1$. Therefore, $|\int{g}d\mu|\leq ||g||_\infty \leq 1$, which yields $||\mu||_W \leq 1$. Reciprocally, consider the constant function $ g \equiv 1$. Then $1 = |\int {g}d\mu| \leq ||\mu||_W$. These two facts proves equation (\ref{simples}). 
\end{proof}

Now we are ready to prove Proposition \ref{l1}.

\begin{proof}(of Proposition \ref{l1} )
	
	In the following, we consider for all $i$, the change of variable $\underline{z} = \sigma_i (\underline{x})$, where $\sigma_i$ is the restriction of $\sigma$ on the cylinder $[0;i]$. We also remark that $\gamma = \gamma _{\underline{x}}$, $\gamma _i := \gamma _{\sigma _i ^{-1}(\underline{x})}$ and $i\underline{x} = \sigma _i ^{-1}(\underline{x})$. Moreover, for a given signed measure $\mu$, we write $c(\underline{x}):= ||\mu|_{\gamma_{\underline{x}}}||_W$. Thus, Lemma \ref{niceformulaac} and equation (\ref{niceformulaa}) yield
	\begin{eqnarray*}
		||\func{F}^{\ast }\mu ||_{\infty} &=&\esssup\{ ||(\func{F}^{\ast }\mu
		)|_{\gamma_{\underline{x}} }||_{W} :\underline{x} \in \Sigma_A^+ \}\\
		&\leq &\esssup\{\sum_{i;A_{i,x_0}=1}{{\ \left|\left| \dfrac{\func{F}%
					_{\gamma _i}^{\ast }\mu |_{\gamma_i}}{%
					J_{m, \sigma_i} (i\underline{x} )}\right| \right|_W}} :\underline{x} \in \Sigma_A^+  \} \\
		&\le&\esssup\{\sum_{i;A_{i,x_0}=1}{{\ \dfrac{||\mu |_{\gamma _i}||_W}{%
					J_{m, \sigma_i} (i\underline{x} )}}} :\underline{x} \in \Sigma_A^+  \}
		\\
		&=&\esssup\{\sum_{i;A_{i,x_0}=1}{{\ \dfrac{c(\sigma_i^{-1}(\underline{x}))}{%
					J_{m, \sigma_i} (i\underline{x})}}\cdot } :\underline{x}\in \Sigma_A^+  \} \\
		&=&||\func {P} _{\sigma} (c) ||_{\infty}\\
		&\leq&||c ||_{\infty}\\ &=&||\mu ||_{\infty}.
	\end{eqnarray*}
\end{proof}

%\begin{proposition}[Lasota-Yorke inequality for $S^{\infty}$] Let $F:\Sigma \longrightarrow \Sigma$ be a map satisfying (...)\marginpar{listar as hipóteses sobre $F$}. Then, there exist $A$, $B_{2}\in \mathbb{R},\lambda <1$ such that, for all $\mu
%\in S^{\infty}$, it holds%
%\begin{equation}
%||\func{F}^{\ast n}\mu ||_{S^{\infty}}\leq A\lambda ^{n}||\mu
%||_{S^{\infty}}+B_{2}||\mu ||_{\infty},\ \ \forall n\geq 1.  \label{xx}
%\end{equation}%
%\label{lasotaoscilation2}
%\end{proposition}

%\begin{proof}

%Firstly, we recall that $\phi _1$ is the marginal density of the disintegration of $\mu$. Precisely, $\phi _1 = \phi _1^+ - \phi _1 ^-$, where $\phi _1 ^+= \dfrac{d\pi _1 ^* \mu ^+}{dm}$ and $\phi _1 ^-= \dfrac{d\pi _1 ^* \mu ^-}{dm}$.

%By equation (\ref{lasotaiiii}), Proposition \ref{l1} and since $|\phi _x|_1 \leq ||\mu||_1$, we have

%\begin{eqnarray*}
%||\func{F}^{\ast n}\mu ||_{S^{1}} &=& |\func{P}_{T}^{n}\phi _x|_{s}  + ||\func{F}^{\ast n}\mu ||_{1} \\ &\leq &  B_3 \beta _2 ^n | \phi _x|_{s}  +  C_2|\phi _x|_{1}  + ||\mu ||_{1}\\ &\leq &  B_3 \beta _2 ^n ||\mu ||_{S^1}  +  (C_2+1) ||\mu ||_{1}.
%\end{eqnarray*}We finish the proof by setting $\lambda=\beta_2$,  $A=B_3$ and $B_2=C_2 +1$. 

%\end{proof}
\subsection{Convergence to equilibrium}

In general, we say that the a transfer operator $\func{F}^*$ has convergence to equilibrium with
at least speed $\Phi$ and with respect to the norms $||\cdot||_{S^{\infty}}$ and $||\cdot||_{\infty}$, if for each $ \mu\in \mathcal{V}$, where 

\begin{equation}\label{Vss}
	\mathcal{V} =\{\mu\in S^{\infty}; \dfrac{d\pi_1^*\mu}{dm}\in \ker(\Pi_{\sigma})\},
\end{equation}it holds
\begin{equation}
	||\func{F}^{\ast  n}\mu||_{\infty}\leq \Phi (n)||\mu||_{S^{\infty}},  \label{wwe}
\end{equation}and $\Phi (n) \longrightarrow 0$ as $n \longrightarrow \infty$.

\begin{remark}\label{rem123}
	We observe that, $\mathcal{V}$ contains all zero average signed measure ($\mu(\Sigma)=0$). Indeed, if $\mu(\Sigma)=0$ (we remember that $\phi _1 = \frac{d(\pi_1^\ast \mu)}{dm}$), then 
	
	\begin{eqnarray*}
		\Pi_\sigma (\phi_1) = \int_{\Sigma^+ _A}{\phi_1}dm =\int_{\Sigma^+ _A}{\dfrac{d(\pi_1^\ast \mu)}{dm}}dm=\int_{\Sigma^+ _A}d(\pi_1^\ast \mu)=\mu(\pi_1^{-1}(\Sigma^+ _A))=\mu(\Sigma)=0.
		%\\ &=& \int_{\Sigma^+ _A}{\dfrac{d(\pi_1^\ast \mu)}{dm}}dm
		%\\ &=& \int_{\Sigma^+ _A}d(\pi_1^\ast \mu)
		%\\ &=& \mu(\pi_1^{-1}(\Sigma^+ _A))
		%\\ &=& \mu(\Sigma)
		%	\\ &=& 0.
	\end{eqnarray*}
\end{remark}

%Next, we prove that $\func{F}^*$ has exponential convergence to equilibrium. This is weaker with respect to spectral gap. However, the spectral gap follows from the above Lasota-Yorke inequality and the convergence to equilibrium. 
%To do it, we need some preliminary lemma and the following is somewhat similar to Lemma \ref{niceformulaac} considering the behaviour of the $||\cdot ||_W$ norm after a contraction. It gives a finer estimate for zero average measures.
The following Lemma is useful to estimate the behaviour of the $||\cdot||_{W}$ under contractions.
\begin{lemma}
	If $F$ satisfies (G1), then for all signed measures $\mu $ on $K$ and for $m$-almost every $\underline{x} \in \Sigma _A^+$, it holds 
	
	\begin{equation*}
		||\func{F}_{\gamma_{\underline{x}} }^\ast \mu ||_{W}\leq \alpha ||\mu ||_{W}+|\mu (K)|
	\end{equation*}%
	($\alpha $ is the rate of contraction of $G$). In particular, if $\mu
	(K)=0$, then%
	\begin{equation*}
		||\func{F}_{\gamma_{\underline{x}} }^\ast \mu ||_{W}\leq \alpha  ||\mu ||_{W}.
	\end{equation*}%
	\label{quasicontract}
\end{lemma}

\begin{proof}
	If $Lip(g)\leq 1$ and $||g||_{\infty }\leq 1$, then $g\circ F_{\gamma }$ is $
	\alpha $-Lipschitz, where $\gamma=\gamma_{\underline{x}}$. Moreover, since $||g||_{\infty }\leq 1$, then $||g\circ
	F_{\gamma }-u ||_{\infty }\leq \alpha $, for some $u$ s.t $|u| \leq 1$. Indeed, consider $z \in K$. We know that $|g \circ F_{\gamma } (z)| \leq 1$. Set $u = g \circ F_{\gamma } (z)$ and let $d$ be the metric of $K$. Thus, we have $$|g\circ
	F_{\gamma } (y)- u | \leq \alpha d(y,z) \leq \alpha \ \mbox{because the diamenter of} \ K \ \mbox{is bounded by} \ 1 $$ and consequently $||g\circ
	F_{\gamma }-u||_{\infty }\leq \alpha $.  
	
	This implies,
	\begin{align*}
		\left\vert \int _{K} {g}d\func{F}_{\gamma }^{\ast }\mu \right\vert & =\left\vert
		\int _{K} {g\circ F_{\gamma }}d\mu \right\vert \\
		& \leq \left\vert \int _{K} {g\circ F_{\gamma }-u }d\mu \right\vert +\left\vert
		\int _{K} {u}d\mu \right\vert \\
		& =\alpha  \left\vert \int _{K} {\frac{g\circ F_{\gamma }-u }{\alpha  }}d\mu
		\right\vert +|u| |\mu (K)|.
	\end{align*}%
	And taking the supremum over $||g||_{\infty }\leq 1$ and $Lip(g)\leq 1$, we
	have $||\func{F}_{\gamma }^{\ast }\mu ||_{W}\leq \alpha ||\mu ||_{W}+\mu
	(K)$. In particular, if $\mu (K)=0$, we get the second part.
\end{proof}

\subsection{${L}^{\infty }$ norms}\label{bggvd}

In this section we consider an $L^{\infty }$ like anisotropic norm. We show
how a Lasota Yorke inequality can be proved for this norm too.

\begin{lemma}
	If $F$ satisfies (G1), for all signed measure $\mu \in
	S^{\infty }$ with marginal density $\phi _1$ it holds 
	\begin{equation}\label{dragon}
		||\func{F}^{\ast }\mu ||_{\infty }\leq \alpha ||\mu
		||_{\infty }+|\phi _1|_{\infty }.
	\end{equation}
\end{lemma}

\begin{proof}
	Denote by $\sigma_i$ the restriction of $\sigma$ on the cylinder $[0;i]$, $\gamma = \gamma _{\underline{x}}, \ \underline{x}=(x_n)_{n\in \mathbb{Z}^+}$, $\gamma _i := \gamma _{\sigma _i ^{-1}(\underline{x})}$ and $i\underline{x} = \sigma _i ^{-1}(\underline{x})$. Applying Lemma \ref{quasicontract} on the third line below and using the facts that $\mu|_{\gamma _i}(K) = \phi
	_1(\sigma_i^{-1}(\underline{x})) $, $\func{P}_\sigma(1)=1$, $|\func{P}_\sigma(\phi_1)|_\infty \leq |\phi_1|_\infty$ we have
	\begin{eqnarray*}
		||(\func {F} ^\ast\mu )|_{\gamma }||_{W} &=&\left|\left| \sum_{i; A_{i,x_0}=1}\frac{%
			\func{F}_{\gamma_i}^{\ast }\mu |_{\gamma _i}}{%
			J_{m,\sigma_{i}}{(i\underline{x})}}\right| \right|_{W}
		\\
		&\leq &\sum_{i; A_{i,x_0}=1}\frac{||\func{F}_{\gamma _i}^{\ast }\mu
			|_{\gamma _i}||_{W}}{J_{m,\sigma_{i}}{(i\underline{x})}} \\
		&\leq &\sum_{i;A_{i,x_0}=1}\frac{\alpha ||\mu |_{\gamma _i}||_{W}+\mu|_{\gamma _i}(K)}{J_{m,\sigma_{i}}{(\sigma_i^{-1}(\underline{x}))}} \\
		&= &\sum_{i; A_{i,x_0}=1}\frac{\alpha ||\mu |_{\gamma _i}||_{W}+\phi
			_1(\sigma_i^{-1}(\underline{x} ) )}{J_{m,\sigma_{i}}{(\sigma_i^{-1}(\underline{x}))}}\\
		&\leq &
		\alpha ||\mu ||_{\infty }\sum_{i;A_{i,x_0}=1}\frac{1}{J_{m,\sigma_{i}}{(\sigma_i^{-1}(\underline{x} ))}}+\sum_{i; A_{i,x_0}=1}\frac{\phi
			_1(\sigma_i^{-1}(\underline{x}))}{J_{m,\sigma_{i}}{(\sigma_i^{-1}(\underline{x}))}}
		\\
		&= &
		\alpha ||\mu ||_{\infty }\func{P}_\sigma(1)(\gamma)+|\func{P}_\sigma(\phi_1)(\gamma)|\\
		&\leq&
		\alpha ||\mu ||_{\infty }+|\func{P}_\sigma(\phi_1)|_\infty \\
		&\leq&
		\alpha ||\mu ||_{\infty }+|\phi_1|_\infty.
	\end{eqnarray*}%
	
\end{proof}Iterating equation (\ref{dragon}) of the previous lemma one obtains

\begin{corollary}\label{nnnn}
	If $F$ satisfies (G1), then for all signed measure $\mu \in
	S^{\infty }$ it holds%
	\begin{equation*}
		||\func{F}^{\ast n}\mu ||_{\infty }\leq \alpha ^{n}||\mu ||_{\infty }+\dfrac{1}{1-\alpha}|\phi _1|_{\infty },
	\end{equation*} where $\phi _1$ is the marginal density of $\mu$.
\end{corollary}

The following proposition shows a regularizing action of the transfer operator with respect to the strong norm. 
Such inequalities are usually called Lasota-Yorke or Doeblin-Fortet inequalities.
\begin{proposition}[Lasota-Yorke inequality for $S^{\infty }$]
	Suppose that $F$ satisfies (G1). Then,
	there are $0<\alpha _1 <1$ and $B_4\in \mathbb{R}$ such that for all $\mu \in S^{\infty}$, it holds%
	\begin{equation*}
		||\func{F}^{\ast n}\mu ||_{S^{\infty }}\leq 2\alpha _1 ^{n}||\mu ||_{S^{\infty
		}}+B_4||\mu ||_{\infty}.
	\end{equation*}
	\label{LYinfty}
\end{proposition}
\begin{proof} Below we shall use the spectral gap for $\func P_{\sigma}$
	
	%We remark that, by equation (\ref{lasotaiiii}) %(N1) it follows %$|\func{P}_{T}^{n}1|_{\infty }\leq H_N (B_3 +C_2),$ for each $n$. Then,
	\begin{eqnarray*}
		||\func{F}^{\ast n}\mu ||_{S^{\infty }} &=&||\func{P}_\sigma ^{n}\phi _1||_{\theta}+||\func{F{^{*n}}}\mu ||_{\infty } \\
		&\leq &\theta^n||\phi _{1}||_{\theta}+C_2|\phi _{1}|_{\infty}+\alpha ^{n}||\mu ||_{\infty }+\frac{1}{1-\alpha}|\phi _{1}|_{\infty }
		\\
		&\leq &2 \alpha_1^n||\mu||_{S^{\infty}} +(C_2+\frac{1}{1-\alpha})|\phi _{1}|_{\infty} %\\&+&[\alpha
		%^{n}H_N (B_3+C_2)||\mu ||_{\infty } +H_N (B_3\beta _{2}^{n}|\phi _{x}|_{s}+C_2|\phi _{x}|_{1})].
		%\\
		%&\leq &[\max (\alpha ,\beta _{2})] ^{n}  [B_3(1+2H_N) + H_NC_2]||\mu ||_{S^{\infty}}+C_2(1+H_N)||\mu ||_{1},
	\end{eqnarray*}%
	where $|\phi_1|_{\infty}\leq ||\mu ||_{\infty}$, $||\phi
	_1||_{\theta}\leq ||\mu ||_{S^{\infty}}$, $||\mu||_{\infty}\le ||\mu||_{S^{\infty}}$ and we set $\alpha_1=\max\{\theta,\alpha\}$.
\end{proof}

\begin{remark}\label{nova}
	By equation (\ref{1}) in Lemma \ref{transformula}, we have
	\begin{equation*}\label{yujhsdf}
		\dfrac{d\pi_1^*(\func{F}^{\ast n}\mu)}{dm} = \func{P}_{\sigma}^n(\phi _1)
	\end{equation*}for all $n \in \mathbb{N}$. Moreover, since $m$ is $\sigma$-invariant, it follows that $\func{P}_{\sigma}(1)=1$ for $m$-almost every $\underline{x} \in \Sigma_A^+$.
	
	Additionally, for any $\underline{x} \in \Sigma_A^+$ and any $h \in L^\infty$, where we denote $i\underline{x} := \sigma _i ^{-1}(\underline{x})$, we have:
	
	\begin{eqnarray*}
		\vert \func{P}_\sigma h (x)\vert &=& \left \vert \sum_{i=1}^N{\dfrac{h(i \underline{x})}{J_{m,\sigma_i}(i\underline{x} )}} \right \vert \\& \leq & |h|_\infty \left\vert \sum_{i=1}^N{\dfrac{1}{J_{m,\sigma_i}(i\underline{x} )}} \right \vert \\& = & |h|_\infty \left\vert \func{P}_\sigma (1) \right\vert \\& = & |h|_\infty. 
	\end{eqnarray*}Thus, we obtain
	
	\begin{equation*}\label{çoaeutrislç}
		|  \func{P}_\sigma h |_\infty \leq |h|_\infty,
	\end{equation*}for all $h \in L^\infty$. Combining the above with the earlier result, we conclude that
	\begin{equation}\label{novo}
		\left | \dfrac{d\pi_1^*(\func{F}^{\ast n}\mu)}{dm} \right|_\infty \leq |\phi_1|_\infty
	\end{equation}for all $\mu \in \mathcal{L}^\infty$ and all $n \in \mathbb{N}$.
\end{remark}

\begin{proposition}[Exponential convergence to equilibrium]\label{5.8}
	Suppose that $F$ satisfies (G1). There exist $D_2\in \mathbb{R}$ and $0<\beta _{1}<1$ such that for every
	signed measure $\mu \in \mathcal{V}$ (see equation (\ref{Vss})), it holds 
	\begin{equation*}
		||\func{F}^{\ast n}\mu ||_{\infty}\leq D_{2}\beta _{1}^{n}||\mu ||_{S^{\infty}},
	\end{equation*}%
	for all $n\geq 1$. \label{quasiquasiquasi}
\end{proposition}

\begin{proof}
	Given $\mu \in \mathcal{V}$ and denoting $\phi _{1}:=\dfrac{\pi_1^\ast\mu}{dm} $ it holds that $\phi \in \Ker(\func {\Pi}_{\sigma})$. Thus, $||\func{P}_{\sigma}^{n}(\phi _{1})||_{\theta}\leq Dr^{n}||\phi _{1}||_{\theta}$ for
	all $n\geq 1$, therefore $||\func{P}_{\sigma}^{n}(\phi _{1})||_{\theta}\leq Dr^{n}||\mu||_{S^{\infty}}$ for all $n\geq 1$.
	%	Consider $l$ and $d \in \{ 0, 1 \}$ such that $n=2l+d$. Thus, $l=\frac{n-d}{2}$ (by Proposition \ref{weakcontral11234}, we have $||\func{F}^{\ast s}%
	%	\mu ||_{\infty}\leq ||\mu ||_{\infty}$, for all $s \in \mathbb{N}$, and $||\mu ||_{\infty}\leq ||\mu
	%	||_{S^{\infty}}$) and by Corollary \ref{nnnn}, it holds (below, set $\beta _1=\max \{\sqrt{r},\sqrt{\alpha}\}$ and $\overline{\alpha }= \dfrac{1}{1-\alpha}$)
	
	Let $l$ and $d \in \{ 0, 1 \}$ be such that $n=2l+d$; hence, $l=\frac{n-d}{2}$. Moreover, Proposition \ref{weakcontral11234} implies that $||\func{F}^{\ast s}%
	\mu ||_{\infty}\leq ||\mu ||_{\infty}$, for all $s \in \mathbb{N}$, and we also have $||\mu ||_{\infty}\leq ||\mu||_{S^{\infty}}$. Then, by equation (\ref{novo}) in Remark \ref{nova} (used in the third inequality below), Corollary \ref{nnnn} (with $1>\beta _1=\max \{\sqrt{r},\sqrt{\alpha}\}$ and $\overline{\alpha }= \dfrac{1}{1-\alpha}$), and since $\beta_1^d \geq \beta_1$ for all $d \in \{ 0, 1 \}$, it follows that

	\begin{eqnarray*}
		||\func{F}^{\ast n}\mu ||_{\infty} &\leq &||\func{F}^{\ast 2l+d } \mu
		||_{\infty} \\
		&\leq &\alpha ^{l}||\func{F}^{\ast l+d} \mu ||_{\infty}+\dfrac{1}{1-\alpha}
		\left|\dfrac{d (\pi _1^\ast (\func{F}^{\ast l+d} \mu ))}{dm}\right|_{\infty} \\
		&\leq &\alpha ^{l}||\mu ||_{\infty}+\dfrac{1}{1-\alpha}|\func{P}_{\sigma}^{l}(\phi
		_{1})|_{\infty}  \\
		&\leq &\alpha ^{l}||\mu ||_{\infty}+\dfrac{1}{1-\alpha}||\func{P}_{\sigma}^{l}(\phi
		_{1})||_{\theta} \\
		&\leq &\alpha ^{l}||\mu ||_{\infty}+\dfrac{1}{1-\alpha}Dr^{l}||\mu ||_{S^{\infty}}\\
		&\leq &(1+\overline{\alpha }D)\beta_1^{-d}\beta_1^n||\mu ||_{S^{\infty}}\\
		&\leq &D_2 \beta_1^n||\mu ||_{S^{\infty}},
	\end{eqnarray*} where $D_2 = \dfrac{1+\overline{\alpha }D}{\beta_1}$, which does not depend on $n$.
\end{proof}

\begin{remark}\label{quantitative}
	We remark that the rate of convergence to equilibrium, $\beta_1$, for the map $F$ found above, is directly related to the rate of contraction, $\alpha$, of the stable foliation, and to the rate of convergence to equilibrium, $r$, of the induced basis map $T$. More precisely, $\beta_1 = \max \{\sqrt{\alpha}, \sqrt{r}\}$. Similarly, we have an explicit estimate for the constant $D_2$, provided we have an estimate for $D$ in the basis map\footnote{It can be difficult to find a sharp estimate for $D$. An approach allowing to find some useful upper estimates is shown in \cite{GNS} }.
\end{remark}

Let $\mu_0$ be the $F$-invariant probability measure constructed by lifting $m$ as an application of Theorem \ref{kjdhkskjfkjskdjf}. By construction, it holds $d(\pi _1^* \mu_0)/dm = 1 \in \mathcal{F}^+_{\theta}(\Sigma_A^{+}).$ With this fact in hands, let us prove theorem \ref{probun}.

\begin{proof}(of Theorem \ref{probun})
	
	Since $F$ that satisfies (G1), Theorem \ref{kjdhkskjfkjskdjf} and Proposition \ref{kjdhkskjfkjskdjff} guarantee the existence of a unique $F$-invariant probability $\mu_0$ such that $\pi _{1*}\mu _0=m$. Moreover, note that $\dfrac{d (\pi _1 ^\ast \mu _0)}{dm} = 1 \in \mathcal{F}^+_{\theta}$. Thus, $||\mu_0|_\gamma||_W=1$ (since it is a probability) and $||\mu _0 ||_\infty =1$. Therefore, $\mu _0 \in S^{\infty}$. 
	
	For the uniqueness, if $\mu _{0}, \mu _{1}\in S^\infty$ are $F$-invariant probabilities, i.e. $\mu_0(\Sigma) = \mu_1(\Sigma)=1$, then by Remark \ref{rem123}, $\mu
	_{0}-\mu _{1}\in \mathcal{V}$. By Proposition \ref{5.8}, $\func{F}^{\ast n}(\mu_{0}-\mu_{1})\rightarrow 0$ in $\mathcal{L}^\infty$. Therefore, $\mu _0 - \mu _1 = 0$.
\end{proof}

\begin{example}\label{dereroidf}
	Let $\mathcal{C}:=\{\varphi _1, \varphi_2, \cdots, \varphi _N\}$ be a finite family of contractions $\phi _i : K \longrightarrow K$, $i=1, \cdots, N$ introduced in Example \ref{kjfhjsfg}. Besides that, let us consider the Hutchinson`s invariant measure of equation (\ref{nvbbjdjfdf}). Since the measure $m \times \mu$ belongs to $S ^\infty$, the Theorem \ref{probun}, yields that
	\begin{equation}\label{oyiuotyu}
		\mu_0 = m \times \mu.
	\end{equation}We believe that, following the ideas of \cite{GLu} and \cite{RRRSTAB}, it is possible to prove statistical stability (computing the modulus of continuity), for the Hutchinson's measure $\mu$, under deterministic perturbations of the IFS.
\end{example}

\section{Spectral gap}\label{wqweer}

In this section, we prove a spectral gap statement for the transfer operator applied to our strong spaces.
For this, we will directly use the properties proved in the previous section, and this will give a kind of constructive proof.
We remark that, like in \cite{GLu}, we cannot appy the traditional Hennion, or Ionescu-Tulcea and Marinescu's approach to our function spaces because there is no compact immersion of the strong space into the weak one. This comes from the fact that we are  considering the same \textquotedblleft dual of Lipschitz\textquotedblright distance in the contracting direction for both spaces.

\begin{proof}(of Theorem \ref{spgap})
	First, let us show there exist $0<\xi <1$ and $M_1>0$ such that, for all $n\geq
	1$, it holds
	\begin{equation}
		||\func{F}^{\ast n}||_{{\mathcal{V}}\rightarrow {\mathcal{V}}}\leq M_1\xi
		^{n}.  \label{quaselawww}
	\end{equation}%
	Indeed, consider $\mu \in \mathcal{V}$ (see equation (\ref{Vss})) s.t. $%
	||\mu ||_{S^{\infty}}\leq 1$ and for a given $n\in \mathbb{N}$ let $m$ and $
	d\in \{0,1\}$ such that $n=2m+d$.
	Thus $m=\frac{n-d}{2}$. By the Lasota-Yorke inequality (Proposition \ref%
	{LYinfty}) we have the uniform bound $||\func{F}^{\ast n}\mu
	||_{S^{\infty}}\leq 2+ B_{4}$ for all $n\geq 1$. Moreover, by Propositions \ref%
	{quasiquasiquasi} and \ref{weakcontral11234} there is some $D_{2}$ such that
	it holds (below, let $\lambda _{0}$ be defined by $\lambda _{0}=\max \{\beta
	_{1},\theta \}$)%
	\begin{eqnarray*}
		||\func{F}^{\ast n}\mu ||_{S^{\infty}} &\leq &2\theta ^{m}||\func{F}^{\ast m+d} \mu ||_{S^{\infty}}+B_{4}||\func{F}^{\ast m+d} \mu ||_{\infty} \\
		&\leq &\theta ^{m}2(2+B_{4})+B_{4}||F^{\ast }{^{m}}\mu ||_{\infty} \\
		&\leq &\theta ^{m}2(2+B_{4})+B_{4}D_{2}\beta _{1}^{m} \\
		&\leq &\lambda _{0}^{m}\left[ 2(2+B_{4})+B_{4}D_{2}\right] \\
		&\leq &\lambda _{0}^{\frac{n-d}{2}}\left[ 2(2+B_{4})+B_{4}D_{2}\right] \\
		&\leq &\left( \sqrt{\lambda _{0}}\right) ^{n}\left( \frac{1}{\lambda _{0}}%
		\right) ^{\frac{d}{2}}\left[ 2(2+B_{4})+B_{4}D_{2}\right] \\
		&=&M_1\xi ^{n},
	\end{eqnarray*}%
	where $\xi =\sqrt{\lambda _{0}}$ and $M_1=\left( \frac{1}{%
		\lambda _{0}}\right) ^{\frac{1}{2}}\left[ 2(2+B_{4})+B_{4}D_{2}\right] $. Thus, we arrive at 
	\begin{equation}
		||(\func{F}^{\ast }|_{_{\mathcal{V}}}){^{n}}||_{S^{\infty}\rightarrow S^{\infty}}\leq
		M_1\xi ^{n}.  \label{just}
	\end{equation}
	
	Now, recall that $\func{F}^{\ast }:S^{\infty}\longrightarrow S^{\infty}$ has a unique
	fixed point $\mu _{0} \in S^\infty$, which is a probability (see Theorem \ref{probun}). We write $\func{P}:S^{\infty}\longrightarrow %
	\left[ \mu _{0}\right] $ ($\left[ \mu _{0}\right] $ is the space spanned by 
	$\mu _{0}$), for the operator defined by $\func{P}(\mu )=\mu (\Sigma)\mu _{0}$. By definition, $%
	\func{P}$ is a projection and $\dim Im(P)=1$. Define the operator%
	\begin{equation*}
		\func{S}:S^{\infty}\longrightarrow \mathcal{V},
	\end{equation*}%
	by%
	\begin{equation*}
		\func{S}(\mu )=\mu -\func{P}(\mu ), \ \ \ \mathnormal{\forall }\ \ \mu \in S^\infty.
	\end{equation*}%
	Thus, we set $\func{N}=\func{F}^{\ast }\circ \func{S}$ and observe that, by
	definition, $\func{P}\func{N}=\func{N}\func{P}=0$ and $\func{F}^{\ast }=\func{%
		P}+\func{N}$. Moreover, $\func{N}^{n}(\mu )=\func{F}^{\ast }{^{n}}(\func{S}%
	(\mu ))$ for all $n \geq 1$. Since $\func{S}$ is bounded and $\func{S}(\mu )\in \mathcal{V}$, we
	get by (\ref{just}), $||\func{N}^{n}(\mu )||_{S^{\infty}}\leq M\xi ^{n}||\mu ||_{S^{\infty}}$, for all $n \geq 1$, where $M=M_1||\func{S}%
	||_{S^{\infty}\rightarrow S^{\infty}}$.
\end{proof}

\begin{remark}\label{quantitative2}
	We remark, the constant $\xi$ for the map $F$, found in Theorem \ref{spgap}, is directly related to the coefficients of the Lasota-Yorke inequality and the rate of convergence to equilibrium of $F$ found before (see Remark \ref{quantitative}). More precisely, $\xi = \max \{\sqrt{\lambda_0} , \sqrt{\beta _1}\}$. We remark that, from the above proof we also have an explicit estimate for $K$ in the exponential convergence, while many classical approaches are not suitable for this. 
\end{remark}

\section{Consequences}\label{kfdjfkjdfeddere}
\subsection{Lipschitz regularity of the disintegration of the invariant measure}
We have seen that a positive measure on $\Sigma_A^+ \times K$, can be disintegrated along the stable
leaves $\mathcal{F}^s$ in a way that we can see it as a family of positive measures on $K$, $\{\mu |_\gamma\}_{\gamma \in \mathcal{F}^s }$. Since there is a one-to-one correspondence between $\mathcal{F}^s$  and $\Sigma_A^+$, this defines a  path
in the metric space of positive measures, $\Sigma_A^+ \longmapsto \mathcal{SB}(K)$, where $\mathcal{SB}(\Sigma_A^+)$ is endowed with the Wasserstein-Kantorovich Like metric (see Definition \ref{wasserstein}). 
It will be convenient to use a functional notation and denote such a path by  
$\Gamma_{\mu } : \Sigma_A^+ \longrightarrow \mathcal{SB}(K)$  defined almost everywhere by $\Gamma_{\mu } (\gamma) = \mu|_\gamma$, where $(\{\mu _{\gamma }\}_{\gamma \in \mathcal{F}^s},\phi_{1})$ is some disintegration for $\mu$.
However, since such a disintegration is defined $\widehat{\mu}$-a.e. $\gamma \in \mathcal{F}^s$, the path $\Gamma_\mu$ is not unique.  For this reason we define more precisely $\Gamma_{\mu } $ as the class of almost everywhere equivalent paths corresponding to $\mu$.

\begin{definition}
	Consider a positive Borel measure $\mu$ and a disintegration  $\omega=(\{\mu _{\gamma_{\underline{x}} }\}_{\underline{x} \in \Sigma _A^+},\phi
	_1)$, where $\{\mu _{\gamma_{\underline{x}} }\}_{\underline{x} \in \Sigma_A^+ }$ is a family of
	probabilities on $\Sigma$ defined $\mu _1$-a.e. $\underline{x} \in \Sigma_A^+$ (where $\mu _1 = \phi _1 m$) and $\phi
	_1:\Sigma_A^+\longrightarrow \mathbb{R}$ is a non-negative marginal density. Denote by $\Gamma_{\mu }$ the class of equivalent paths associated to $\mu$ 
	%(of positive measures on $I$) $\Gamma_{\mu
		%}:I\longrightarrow \mathcal{SB}(I)$ defined $\mu _x$-a.e. $\gamma \in I$ by 
	\begin{equation*}
		\Gamma_{\mu }=\{ \Gamma^\omega_{\mu }\}_\omega,
	\end{equation*}
	where $\omega$ ranges on all the possible disintegrations of $\mu$ and $\Gamma^\omega_{\mu }: \Sigma_A ^+\longrightarrow \mathcal{SB}(\Sigma_A^+)$ is the map associated to a given disintegration, $\omega$:
	$$\Gamma^\omega_{\mu }(\underline{x})=\mu |_{\gamma} = \pi _{\gamma, 2} ^\ast \phi _1(\underline{x})\mu _\gamma,$$where $\gamma = \gamma_{\underline{x}}$.
\end{definition}Let us call the set on which $\Gamma_{\mu }^\omega $ is defined by $I_{\Gamma_{\mu }^\omega } \left( \subset \Sigma _A ^+\right)$ and in the following definition we set $\gamma _1 = \gamma _{\underline{x}^1}$ and $\gamma _1 = \gamma _{\underline{x}^2}$ where $\underline{x}^1, \underline{x}^2 \in \Sigma _A^+$.

%In the following, when no ambiguity is possible we will consider informally $\Gamma_{\mu }$ itself as a path.

%\footnote{Remark that to a measure many different paths and sets $I_{\Gamma_{\mu }}$ may be associated, but they coincide almost %everywhere.}. 

\begin{definition}Given a disintegration $\omega$ of $\mu$ and its functional representation $\Gamma_{\mu }^\omega $ we define the \textbf{Lipschitz constant of $\mu$ associated to $\omega$} by

	\begin{equation}\label{Lips1}
		|\mu|_\theta ^\omega := \esssup _{\underline{x}^1, \underline{x}^2 \in I_{\Gamma_{\mu }^\omega}} \left\{ \dfrac{||\mu|_{\gamma _1}- \mu|_{\gamma _2}||_W}{d_\theta (\underline{x}^1,\underline{x}^2)}\right\},
	\end{equation}where the essential supremum is taken with respect to the Markov measure $m$. By the end, we define the \textbf{Lipschitz constant} of the positive measure $\mu$ by

	\begin{equation}\label{Lips2}
		|\mu|_\theta :=\displaystyle{\inf_{ \Gamma_{\mu }^\omega \in \Gamma_{\mu } }\{|\mu|_\theta ^\omega\}}.
	\end{equation}
	
	\label{Lips3}
\end{definition}

\begin{remark}
	When there is no scope of confusion, to simplify the notation, we denote $\Gamma_{\mu }^\omega (\underline{x})$ just by $\mu |_{\gamma}$.
\end{remark}

\begin{definition} \label{erfcscvdsd}
	From the Definition \ref{Lips3} we define the set of Lipschitz positive measures $\mathcal{L} _\theta^{+}$ as
	\begin{equation}
		\mathcal{L}_\theta ^{+}=\{\mu \in \mathcal{AB}:\mu \geq 0,|\mu |_\theta <\infty \}.
	\end{equation}
\end{definition}For the next lemma, for a given path, $\Gamma _\mu$ which represents the measure $\mu$, we define for each $\gamma \in I_{\Gamma_{\mu }^\omega }\subset \Sigma _A ^+$, the map

\begin{equation}
	\mu _F(\underline{x}) := \func{F_\gamma^*}\mu|_\gamma,
\end{equation}where $F_\gamma :K \longrightarrow K$ is defined as $F_\gamma (y) = \pi_2 \circ F \circ {(\pi _2|_\gamma)} ^{-1}(y)$, $\pi_2 : \Sigma _A ^+\times K \longrightarrow \Sigma _A ^+\times K$ is the projection $\pi_2(x,y)=y$ and $\gamma = \gamma_{\underline{x}}$. In the proof of the following lemma, we also use the notation $\gamma _1 = \gamma _{\underline{x}^1}$ and $\gamma _1 = \gamma _{\underline{x}^2}$ where $\underline{x}^1, \underline{x}^2 \in \Sigma _A^+$. 

%\newpage
\begin{lemma}\label{hdhfjdf} If $F$ satisfies (G1) and (G2), then for all representation $\Gamma _\mu$ of a positive measure $\mu \in \mathcal{L}_\theta ^{+}$, it holds
	\begin{equation}
		|\mu_F|_\theta \leq |\mu|_\theta ^\omega + H ||\mu||_\infty.
	\end{equation}
\end{lemma}
\begin{proof}
	\begin{eqnarray*}
		||\func{F_{\gamma _1}^*}\mu|_{\gamma_1}- \func{F_{\gamma _2}^*}\mu|_{\gamma_2}||_W &\leq & ||\func{F_{\gamma _1}^*}\mu|_{\gamma_1}- \func{F_{\gamma _1}^*}\mu|_{\gamma_2}||_W \\&+&||\func{F_{\gamma _1}^*}\mu|_{\gamma_2}- \func{F_{\gamma _2}^*}\mu|_{\gamma_2}||_W 
		\\ &\leq & ||\mu|_{\gamma_1}- \mu|_{\gamma_2}||_W +||\func{F_{\gamma _1}^*}\mu|_{\gamma_2}- \func{F_{\gamma _2}^*}\mu|_{\gamma_2}||_W
		\\ &\leq & ||\mu|_{\gamma_1}- \mu|_{\gamma_2}||_W \\&+&\int{d_2(G(\underline{x}^1, y),G(\underline{x}^2,y))} d(\mu|{\gamma _2}) (y).
	\end{eqnarray*}	Thus,

	\begin{eqnarray*}
		\dfrac{||\func{F_{\gamma _1}^*}\mu|_{\gamma_1}- \func{F_{\gamma _2}^*}\mu|_{\gamma_2}||_W}{d_\theta (\underline{x}^1,\underline{x}^2)} 
		&\leq &  \dfrac{||\mu|_{\gamma_1}- \mu|_{\gamma_2}||_W}{d_\theta (\underline{x}^1,\underline{x}^2)} \\ &+& \dfrac{\esssup _yd_2(G(\underline{x}^1,y), G(\underline{x}^2,y))}{d_\theta (\underline{x}^1,\underline{x}^2)} \int {1}d(\mu|{\gamma _2}) (y)
		\\ &\leq &  |\mu|_\theta ^\omega + H||\mu |_{\gamma _2}||_W.
	\end{eqnarray*}Therefore,
	
	\begin{equation}
		|\mu_F|_\theta \leq |\mu|_\theta ^\omega + H ||\mu||_\infty.
	\end{equation}
	
\end{proof}

For the next proposition and henceforth, for a given path $\Gamma _\mu ^\omega \in \Gamma_{ \mu }$ (associated with the disintegration $\omega = (\{\mu _\gamma\}_\gamma, \phi _1)$, of $\mu$, unless written otherwise, we consider the particular path $\Gamma_{\func{F}^{\ast}\mu} ^\omega \in \Gamma_{\func{F}^{\ast}\mu}$ defined by the Proposition \ref{niceformulaab}, by the expression

\begin{equation}
	\Gamma_{\func{F}^{\ast}\mu} ^\omega (\underline{x})=\sum_{i; A_{i,x_0}=1}\func{F}^{\ast}%
	_{\gamma _i} \Gamma _\mu ^\omega (i\underline{x})g _i(i\underline{x})\ \ m%
	\mathnormal{-a.e.}\ \ \underline{x} \in \Sigma ^+ _A,  \label{niceformulaaareer}
\end{equation}where $g_i(\underline{x})=\dfrac{1}{J_{m,\sigma_{i}}{(\underline{x})}}$ for all $x \in [0;i]$ and all $i=1, \cdots, N$. We recall that, $\Gamma_{\mu} ^\omega (\underline{x}) = \mu|_\gamma:= \pi_{2}^\ast (\phi_1(\underline{x})\mu _\gamma)$ and in particular $\Gamma_{\func{F}^*\mu} ^\omega (\underline{x}) = (\func{F}^*\mu)|_\gamma = \pi_{2}^*(\func{P}_\sigma\phi_1(\underline{x})\mu _\gamma)$, where $\phi_1 = \dfrac{d \pi _{1}^* \mu}{dm}$, $\gamma = \gamma_{\underline{x}}$, $i\underline{x}:=\sigma_i^{-1}(\underline{x})$ and $\func{P}_\sigma$ is the Perron-Frobenius operator of $f$.

\begin{proposition}\label{kjsdkjduidf}
	If $F$ satisfies (G1) and (G2), then for all representation $\Gamma _\mu$ of a positive measure $\mu \in \mathcal{L}_\theta ^{+}$, it holds 
	
	\begin{equation}
		|\Gamma^{\omega}_{\func{F}^*\mu}|_\theta ^\omega \leq \theta  |\mu|_\theta ^\omega + C_1 ||\mu||_\infty,
	\end{equation}where $C_1 = \max \{ H\theta +\theta N |g|_\theta , 2\}$.
\end{proposition}
\begin{proof}
	Below we use the notation $g_j(\underline{x})=\dfrac{1}{J_{m,\sigma_{j}}{(\underline{x})}}$ and for a given sequence $\underline{x} = (x_i)_{i \in \mathbb{Z}^+}  \in \Sigma ^+ _{A}$ we denote by $j\underline{x} = (z_i)_{i \in \mathbb{Z}^+}$ the sequence defined by $z_0 = j$ and $z_i = x_{i-1}$ for all $i \geq 1$. In this case, it is easy to see that $ d_\theta(j\underline{x}^1, j\underline{x}^2) = \theta d_\theta (\underline{x}^1, \underline{x}^2)$ for all $\underline{x}^1, \underline{x}^2 \in \Sigma _A ^+$.
	
	We divide the proof of the proposition in two parts. First we consider two sequences $\underline{x}^1 = (x_i)_{i \in \mathbb{Z}^+}$ and $\underline{x}^2 = (z_i)_{i \in \mathbb{Z}^+}$ such that $x_0=z_0$. By Proposition \ref{niceformulaab} (and equation (\ref{niceformulaaareer})) and setting $\gamma^1 = \gamma_{\underline{x}^1}$, $\gamma^2 = \gamma_{\underline{x}^2}$, $\gamma_i := \gamma_{i \underline{x}}$ for all $\underline{x} \in \Sigma_A^+$ and all $i=1, \cdots, N$, we have
	
	\begin{eqnarray*}
		||\Gamma_{\func{F}^*\mu} ^\omega (\underline{x}^1)- \Gamma_{\func{F}^*\mu} ^\omega (\underline{x}^2) ||_W &\leq& \sum_{j: A_{j,x_0}=1}{||\func{F_{\gamma ^1_j}^*}\mu|_{\gamma ^1_j}||_W}|g_j(j\underline{x}^1)-g_j(j\underline{x}^2)|\\&+& \sum _{j: A_{j,x_0}=1}{|g_j(j\underline{x}^2)|||\func{F_{\gamma ^1_j}^*}\mu|_{\gamma ^1_j}-\func{F_{\gamma ^2_j}^*}\mu|_{\gamma ^2_j}||_W}.
	\end{eqnarray*}Thus,

	\begin{eqnarray*}
		\dfrac{||\Gamma_{\func{F}^*\mu} ^\omega (\underline{x}^1)- \Gamma_{\func{F}^*\mu} ^\omega (\underline{x}^2)||_W}{d_\theta(\underline{x}^1,\underline{x}^2)} 
		&\leq& \theta \sum_j{||\func{F_{\gamma ^1_j}^*}\mu|_{\gamma ^1_j}||_W}\dfrac{|g_j(j\underline{x}^1)-g_j(j\underline{x} ^2)|}{\theta d_\theta(\underline{x}^1, \underline{x}^2)} \\&+& \theta\sum_j{|g_j(j\underline{x}^2)|\dfrac{||\func{F_{\gamma ^1_j}^*}\mu|_{\gamma ^1_j}-\func{F_{\gamma ^2_j}^*}\mu|_{\gamma ^2_j}||_W}{\theta d_\theta(\underline{x}^1, \underline{x}^2)}}
		\\&\leq& \theta \sum_j{||\func{F_{\gamma ^1_j}^*}\mu|_{\gamma ^1_j}||_W}\dfrac{|g_j(j\underline{x}^1)-g_j(j\underline{x}^2)|}{ d_\theta(j\underline{x}^1, j\underline{x}^2)} \\&+& \theta\sum_j{|g_j(j\underline{x}^2)|\dfrac{||\func{F_{\gamma ^1_j}^*}\mu|_{\gamma ^1_j}-\func{F_{\gamma ^2_j}^*}\mu|_{\gamma ^2_j}||_W}{ d_\theta(j\underline{x}^1, j\underline{x}^2)}}
		\\&:=&S_1+S_2.
	\end{eqnarray*}
	Note that since $||\func{F}_{\gamma}^{\ast}\mu|_{\gamma}||_{W}\le ||\mu|_{\gamma}||_W$
	\begin{eqnarray*} S_1&:=&\theta \sum_j{||\func{F_{\gamma ^1_j}^*}\mu|_{\gamma ^1_j}||_W}\dfrac{|g_j(j\underline{x}^1)-g_j(j\underline{x}^2)|}{\theta d_\theta(\underline{x}^1, \underline{x}^2)}\\ &\le&  \theta \sum_{j}||\mu|_{\gamma ^1_j}||_{W}|g|_{\theta} \le \theta N |g|_{\theta}||\mu||_{\infty}
	\end{eqnarray*}
	and 
	\begin{eqnarray*}S_2 &:=& \theta\sum_j{|g_j(j\underline{x}^2)|\dfrac{||\func{F_{\gamma ^1_j}^*}\mu|_{\gamma ^1_j}-\func{F_{\gamma ^2_j}^*}\mu|_{\gamma ^2_j}||_W}{ d_\theta(j\underline{x}^1, j\underline{x}^2)}}\\ &\le & \theta \sum_{j}g_j(j\underline{x}^2)|\mu_F|^{\omega}_{\theta}, \ \mbox{remember the definition of} \ \mu_F \ \mbox{in the equation (7.4).} 
	\end{eqnarray*}
	Applying Lemma \ref{hdhfjdf}, we get
	\begin{eqnarray*}
		\dfrac{||\Gamma_{\func{F}^*\mu} ^\omega (\underline{x}^1)- \Gamma_{\func{F}^*\mu} ^\omega (\underline{x}^2)||_W}{d_\theta(\underline{x}^1, \underline{x}^2)} 
		&\leq&  \theta N |g|_\theta||\mu||_{\infty}+ \theta\sum_j{|g_j(j\underline{x}^2)||\mu_F|_\theta ^\omega}
		\\	&\leq&  \theta N |g|_\theta||\mu||_{\infty}+ \theta\sum_j{|g_j(j\underline{x}^2)|}\left(|\mu|_\theta^\omega + H ||\mu||_\infty \right)
		\\	&=&  \theta  |\mu|_{\theta} ^\omega+ \left(H\theta + N \theta |g|_\theta\right) ||\mu||_\infty .
	\end{eqnarray*}Taking the supremum over $\underline{x}^1$ and $\underline{x}^2$ we get 
	
	\begin{equation}
		|\func{F}^*\mu|_\theta ^\omega \leq \theta  |\mu|_\theta ^\omega + \left[H\theta +\theta N |g|_\theta \right] ||\mu||_\infty,
	\end{equation}where $H$ was defined by equation (\ref{hhfksdjfksdfsdfsd}). 
	
	In the remaining case, where $x_0 \neq y_0$ it holds $d_\theta (\underline{x}^1,\underline{x}^2) \geq 1$. Then, by Lemma \ref{niceformulaac} we have
	
	\begin{eqnarray*}
		\dfrac{||\Gamma_{\func{F}^*\mu} ^\omega (\underline{x}^1)- \Gamma_{\func{F}^*\mu} ^\omega (\underline{x}^2) ||_W}{d_\theta(\underline{x}^1, \underline{x}^2)} &\leq & ||(\func{F}^*\mu)|_{\gamma^1}- (\func{F}^*\mu)|_{\gamma ^2} ||_W \\&\leq & 2 ||\mu ||_\infty.
	\end{eqnarray*} We finish the proof by setting $C_1 = \max \{ H\theta +\theta N |g|_\theta , 2\}$.

\end{proof}Iterating the inequality given by the previous proposition we immediately get

\begin{theorem}\label{hdjfhsdfjsd}
	If $F$ satisfies (G1) and (G2), then for all representation $\Gamma _\mu$ of a positive measure $\mu \in \mathcal{L}_\theta ^{+}$ and all $n \geq 1$, it holds 
	
	\begin{equation}\label{uerjerh}
		|\Gamma^{\omega}_{\func{F{^*}}^n\mu}|_\theta ^\omega \leq \theta^n |\mu|_\theta ^\omega + \dfrac{C_1}{1- \theta} ||\mu||_\infty,
	\end{equation}where $C_1$ was defined in Proposition \ref{kjsdkjduidf} by $C_1 = \max \{ H\theta +\theta N |g|_\theta , 2\}$.
\end{theorem}

\begin{remark}\label{kjedhkfjhksjdf}
	Taking the infimum (with respect to $\omega$) on both sides of inequality \ref{uerjerh} we get for each $\mu \in \mathcal{L}_\theta ^{+}$
	
	\begin{equation}\label{kuyhj}
		|\func{F{^*}}^n\mu|_\theta  \leq \theta^n |\mu|_\theta + \dfrac{C_1}{1- \theta} ||\mu||_\infty,
	\end{equation}for all $n \geq 1$.
\end{remark}

\begin{remark}\label{riirorpdf}
	For a given probability measure $\nu$ on $K$. Denote by $m_1$, the product $m_1=m \times \nu$, where $m$ is the Markov measure fixed in the subsection \ref{sec1}. Besides that, consider its trivial disintegration $\omega_0 =(\{m_{1\gamma}  \}_{\gamma}, \phi_1)$, given by $m_{1,\gamma} = \func{\pi _{2,\gamma}^{-1}{^*}}\nu$, for all $\gamma$ and $\phi _1 \equiv 1$. According to this definition, it holds that 
	\begin{equation*}
		m_1|_\gamma = \nu, \ \ \forall \ \gamma.
	\end{equation*}In other words, the path $\Gamma ^{\omega _0}_{m_1}$ is constant: $\Gamma ^{\omega _0}_{m_1} (\gamma)= \nu$ for all $\gamma$. In particular, this implies that 
	\begin{equation}\label{oiyiye}
		|m_1|_\theta ^{\omega_0} = 0.
	\end{equation}Moreover, for each $n \in \mathbb{N}$, let $\omega_n$ be the particular disintegration for the measure $\func{F{^\ast }}^nm_1$, defined from $\omega_0$ as an application of Lemma \ref{transformula} and consider the path $\Gamma^{\omega_{n}}_{\func{F}{^*}{^n} m_1}$ associated with this disintegration. 
	
	A finite string $\underline{a}=(a_0,a_1,...,a_{n-1})\in \{1,2,...,N\}^n$ is \textit{allowed} if $A_{a_{i-i},a_{i}}=1$ for $i=1,...,n-1$. The set of allowed strings of length $n$ is denoted by $\mathcal{A}_n$. If $\underline{a} \in \mathcal{A}_n$, then $\sigma^{-n}{\underline{a}}$ is well-defined on the set ${\underline{x}; A_{a_{n-1}, x_0} = 1}$, taking values on the cylinder $[0; \underline{a}] := {\underline{x}; x_0 = a_0, x_1 = a_1, \dots, x_{n-1} = a_{n-1}}$, by $$\sigma^{-n}_{\underline{a}}=\sigma^{-1}_{a_0}\circ \sigma^{-1}_{a_1}\circ...\circ\sigma^{-1}_{a_{n-1}},$$that is, $$\sigma^{-n}_{\underline{a}}(\underline{x})=\underline{a}\underline{x}.$$By Proposition \ref{niceformulaab}, if we define $$\func{F}^{\ast n}_{\gamma_{\underline{a}}}:=\func{F}^{\ast}%
	_{\gamma _{a_0}}(\func{F}^{\ast}_{\gamma_{a_1}}( \dots (\func{F}^{\ast}_{\gamma_{a_{n-1}}})))$$ and $$J_{m,\underline{a}}(\sigma_{\underline{a}}^{-n}(\underline{x})):=J_{m,\sigma_{a_0}}(\sigma^{-1}_{a_0}(\underline{x}))J_{m,a_1}(\sigma^{-2}_{a_0a_1}(\underline{x}))...J_{m,a_{n-1}}( \sigma_{\underline{a}}^{-n}(\underline{x}))$$ we have
	\begin{equation}
		\Gamma^{\omega_{n}}_{\func{F}^{\ast n } m_1} (\underline{x})  =\sum_{\underline{a}x_0\in \mathcal{A}_{n+1}}\dfrac{\func{F}^{\ast n}_{\gamma_{\underline{a}}}\nu}{J_{m,\underline{a}}(\sigma_{\underline{a}}^{-n}(\underline{x}))}\ \ m-\hbox{a.e} \ \ \underline{x}=(x_j)_{j\in \mathbb{Z}^+} \in \Sigma ^+ _A. \label{oiyiy}
	\end{equation} Applying Theorem \ref{hdjfhsdfjsd} and equation (\ref{oiyiy}) on the path $\Gamma^{\omega_n}_{\func{F^{\ast n}} m_1}$ defined by equation (\ref{oiyiy}), we obtain
	
	\begin{equation}\label{nbmjsdfjf}
		|\Gamma^{\omega_{n}}_{\func{F{^\ast }}^n m_1}|_\theta ^\omega \leq \dfrac{C_1}{1-\theta}, \ \forall \ n \geq 1.
	\end{equation}

\end{remark}

\begin{proof}(of Theorem \ref{regu})
	
	%First of all, it is not hard to prove that if $\Gamma^\omega_{n}:\widehat{I}\longrightarrow \mathcal{SB}([0,1])$ is a sequence of paths which converges
	%to $\Gamma_{\mu _{0}}^\omega:\widehat{I}\longrightarrow \mathcal{SB}([0,1])$ pointwise on a
	%full measure set $\widehat{I}\subset I$, then for every fixed partition $\mathcal{P%
		%}=\{x_{0},\cdots ,x_{n}\}\subset \widehat{I}$ it holds 
	%\begin{equation*}
	%\lim_{n\longrightarrow \infty }{\Var(\Gamma^\omega_{{n}},\mathcal{P})}=\Var(\Gamma^\omega_{\mu _{0}},%
	%\mathcal{P}).
	%\end{equation*}
	
	According to Theorem \ref{probun}, let $\mu _{0}\in S^{\infty}$ be the unique $F$-invariant probability measure in $S^\infty$. 	Consider the path $\Gamma^{\omega_n}_{\func{F{^\ast }^n}m_1}$, defined in Remark \ref{riirorpdf},  which represents the measure $\func{F{^\ast }}^nm_1$. By Theorem \ref{spgap}, these iterates converge to $\mu _{0}$
	in $\mathcal{L}^{\infty }$. It means that the sequence $\{\Gamma_{\func{F ^{\ast n}}(m_1)} ^\omega\}_{n}$ converges $m$-a.e. to $\Gamma_{\mu _{0}}^\omega\in \Gamma_{\mu_0 }$ (in $\mathcal{SB}(K)$ with respect to the metric given by Definition \ref{wasserstein}),  where $\Gamma_{\mu _{0}}^\omega$ is a path given by the Rokhlin Disintegration
	Theorem and $\{\Gamma_{\func{F ^{\ast n}}(m_1)} ^\omega\}_{n}$ is given by Remark \ref{niceformulaab}. It implies that $\{\Gamma_{\func{F ^{\ast n}}(m_1)} ^\omega\}_{n}$ converges pointwise to $\Gamma_{\mu _{0}}^\omega$ on a full measure set $\widehat{I}\subset \Sigma _A^+$. Let us denote $%
	\widehat{\Gamma^\omega_{n}}=\Gamma^\omega_{\func{F^{\ast n }}(m_1)}|_{%
		\widehat{I}}$ and $\widehat{\Gamma^\omega_{\mu _{0}}}=\Gamma^\omega _{\mu _{0}}|_{\widehat{I}}$. Since $\{\widehat{\Gamma^\omega_{n}} \}_n $ converges pointwise to $\widehat{\Gamma^\omega_{\mu _{0}}}$, for all $\underline{x}^1, \underline{x}^2 \in \widehat{I}$ it holds 
	
	\begin{eqnarray*}
		\lim _{n \longrightarrow \infty} {\dfrac{||\widehat{\Gamma^\omega_{n}} (\underline{x}^1)- \widehat{\Gamma^\omega_{n}}(\underline{x}^2)||_W}{d_\theta(\underline{x}^1,\underline{x}^2)}}  &=& {\dfrac{||\widehat{\Gamma^\omega_{\mu _0}} (\underline{x}^1)- \widehat{\Gamma^\omega_{\mu _0}}(\underline{x}^2)||_W}{d_\theta(\underline{x}^1,\underline{x}^2)}}.
	\end{eqnarray*} On the other hand, by equation (\ref{nbmjsdfjf}), we have $\dfrac{||\widehat{\Gamma^\omega_{n}} (\underline{x}^1)- \widehat{\Gamma^\omega_{n}}(\underline{x}^2)||_W}{d_\theta(\underline{x}^1, \underline{x}^2)}\leq |\widehat{\Gamma^\omega_{n}}|_\theta ^\omega \leq \dfrac{C_1}{1-\theta}$ for all $n\geq 1$. Then $|\widehat{\Gamma^\omega_{\mu _0}}|^\omega _\theta \leq  \dfrac{C_1}{1-\theta}$ and hence  $|\mu_0| _\theta \leq  \dfrac{C_1}{1-\theta}$.
	
\end{proof}

\begin{remark}
	We remark that, Theorem \ref{regu} is an estimation of the regularity of the disintegration of $\mu _{0}$. Similar results, for other sort of skew products are presented in \cite{GLu}, \cite{BM}, \cite{RRR} and \cite{RRRSTAB}. 
\end{remark}

\section{Exponential decay of correlations}\label{decayy}

\label{decay} In this section, we will show how Theorem \ref{spgap} implies an exponential rate of convergence for the limit $$\lim {C_n(f,h)}=0,$$where $$C_n(f,h):=\left| \int{(h \circ F^n )  f}d\mu_0 - \int{h  }d\mu_0 \int{f  }d\mu_0 \right|,$$ $h: \Sigma_A^+ \times K \longrightarrow \mathbb{R} $ is a Lipschitz function and $f \in \Theta _{\mu _0}$. The set $\Theta _{\mu _0}$ is defined as 
\begin{equation}\label{oiityncfd}
	\Theta _{\mu _0}:= \{ f: \Sigma_A^+ \times K \longrightarrow \mathbb{R}; f\mu_0 \in S^\infty\}, 
\end{equation}where the signed measure $f\mu_0$ is defined by $f\mu_0(E):=\int _E{f}d\mu_0$ for all measurable set $E$. On $\Theta _{\mu _0}$ we define the norm $|| \cdot ||_\Theta:\Theta _{\mu _0} \longrightarrow \mathbb{R}$ by 

\begin{equation}
	||f ||_{\Theta_{\mu_0}}:= ||f\mu_0||_{S^{\infty}}.
\end{equation}%\marginpar{verificar se essa norma foi utilizada em algum lugar. Caso n tenha sido usada, excluir}

Recall that $\widehat{\mathcal{F}}(\Sigma^+_A\times K)$ the set of real Lipschitz functions, $f:\Sigma^+_A \times K \longrightarrow \mathbb{R}$, with respect to the metric $d_\theta + d$. For such a function we denote by $L_{\Sigma}(f)$ its Lipschitz constant and $||h||_{Lip}=||h||_{\infty}+L_{\Sigma}(h)$, where $\Sigma=\Sigma^{+}_A\times K$.

\begin{proposition}\label{kkkskdjd}
	Consider $F$ satisfying (G1) and (G2). For all Lipschitz function $h:\Sigma \longrightarrow \mathbb{R}$ and all $f \in \Theta _{\mu _0}$, it holds $$\left| \int{(h \circ F^n )  f}d\mu_0 - \int{h  }d\mu_0 \int{f  }d\mu_0 \right| \leq M \xi ^{n}  ||f ||_{\Theta_{\mu_0}}  ||h||_{Lip},   \ \ \forall n \geq 1,$$where $\xi$ and $K$ are from Theorem \ref{spgap}. 
\end{proposition}

\begin{proof}
	
	Let $h: \Sigma \longrightarrow \mathbb{R} $ be a Lipschitz function and $f \in \Theta _{\mu _0} ^1$. By Theorem \ref{spgap}, we have
	\begin{eqnarray*}
		C_n(f,h) &=& \left| \int{h  }d \func{F^*}{^n} (f\mu_0) - \int{h  }d\func{P}(f\mu_0) \right|
		\\&=& \left|\left|  \func{F^*}{^n} (f\mu_0) - \func{P}(f\mu_0) \right|\right|_W \max\{L_{\Sigma} (h), ||h||_\infty\}
		\\&=& \left|\left|  \func{N}{^n}(f\mu_0) \right|\right|_W \max\{L_{\Sigma}(h), ||h||_\infty\}
		\\&\leq & \left|\left|  \func{N}{^n}(f\mu_0) \right|\right|_{S^\infty} \max\{L_{\Sigma} (h), ||h||_\infty\}
		\\&\leq &  M \xi ^{n}  ||f \mu _0||_{S^{\infty}} ||h||_{Lip}
		\\&= &  M \xi ^{n}  ||f ||_{\Theta_{\mu _0}} ||h||_{Lip}.
	\end{eqnarray*}
	
\end{proof}

\subsection{From a Space of Measures to a Space of Functions \label{last123}}

In this section, we will show how the regularity of the $F$-invariant measure, given by Theorem \ref{regu}, implies that $\widehat{\mathcal{F}}(\Sigma) \subset \Theta _{\mu _0}$. Where, $\widehat{\mathcal{F}}(\Sigma)$ is the set of real Lipschitz functions, $h:\Sigma \longrightarrow \mathbb{R}$ defined in the previous in Section \ref{sotmr} and $\Theta _{\mu _0}$ is defined in equation (\ref{oiityncfd}).

The proof of the following lemma can be found in Lemma 8.23 of \cite{GLu} or Lemma 8.1 of \cite{RRR}.
\begin{lemma}\label{hdgfghddsfg}
	Let $(\{\mu_{0, \gamma}\}_\gamma, \phi_1)$ be the disintegration of $\mu _0$ along the partition $\mathcal{F}^s:=\{\gamma:= \{ \underline{x}\} \times K: \underline{x} \in \Sigma _A^+\}$, and for a $\mu_0$-integrable function $h:\Sigma\longrightarrow \mathbb{R}$, denote by $\nu$ the measure defined by $\nu:=h\mu_0$ (where $h\mu_0(E) := \int _E {h}d\mu _0$). If $(\{\nu_{ \gamma}\}_\gamma, \widehat{\nu} )$ is the disintegration of $\nu$, where $\widehat{\nu}:=\pi_1{_*} \nu$, then $\widehat{\nu} \ll m$ and $\nu _\gamma \ll \mu_{0, \gamma}$. Moreover, denoting $\overline{h}:=\dfrac{d\widehat{\nu}}{dm}$, it holds 
	\begin{equation}\label{fjgh}
		\overline{h}( \underline{x})=\int_{K}{h( \underline{x}, y)}d(\mu_0|_\gamma)(y),
	\end{equation} and for $\widehat{\nu}$-a.e. $ \underline{x} \in \Sigma _A^+$
	
	\begin{equation}\label{gdfgdgf}
		\dfrac{d\nu _{ \gamma}}{d\mu _{0, \gamma}}(y) =
		\begin{cases}
			\dfrac{h|_\gamma (y)}{\int{h|_\gamma(y)}d\mu_{0,\gamma}(y)} , \ \hbox{if} \  \underline{x} \in B ^c \\
			0, \ \hbox{if} \  \underline{x} \in B,
		\end{cases} \hbox{for all} \ \ y \in K,
	\end{equation}
	where $B :=  \overline{h} ^{-1}(0)$ and $\gamma=\gamma_{\underline{x}}$.%$=\{ \underline{x} \}\times K.$
\end{lemma}

\begin{proposition}\label{ksjdhsdd}
	If $F$ satisfies (G1) and (G2), then $\widehat{\mathcal{F}}(\Sigma) \subset \Theta _{\mu _0}$.
\end{proposition}
\begin{proof}
	For a given $f \in \widehat{\mathcal{F}}(\Sigma)$, denote by $k:=\max \{L_{\Sigma}(f), ||f||_\infty \}$. Consider $\underline{x}^1$ and $\underline{x}^2$ points of $\Sigma^+_A$ and $\gamma_1=\gamma_{\underline{x}^1}$ and $\gamma_2=\gamma_{\underline{x}^2}$. Thus,

	\begin{eqnarray*}
		|\overline{f}(\underline{x}^2) - \overline{f}(\underline{x}^1)|&=& \left|\int_K{f(\underline{x}^2, y)}d(\mu_0|_{\gamma_2}) - \int_K{f(\underline{x}^1, y)}d(\mu_0|_{\gamma_1})\right|
		\\& \leq & \left|\int_K{f(\underline{x}^2, y)}d(\mu_0|_{\gamma_2}) - \int_K{f(\underline{x}^1, y)}d(\mu_0|_{\gamma_2})\right| \\&+& \left|\int_K{f(\underline{x}^1, y)}d(\mu_0|_{\gamma_2}) - \int_K{f(\underline{x}^1, y)}d(\mu_0|_{\gamma_1})\right|
		\\& \leq & \int_K{\left|f(\underline{x}^2, y) -f(\underline{x}^1, y) \right| }d(\mu_0|_{\gamma_2}) 
		\\&+& k \left|\int_K{\dfrac{f(\underline{x}^1, y)}{k}}d(\mu_0|_{\gamma_1} - \mu_0|_{\gamma_2})\right|
		\\& \leq &L_{\Sigma}(f) d_\theta (\underline{x}^2,\underline{x}^1)\int_K{1 }d(\mu_0|_{\gamma_2}) 
		+ k \left|\left|\mu_0|_{\gamma_1} - \mu_0|_{\gamma_2}\right|\right|_W
		\\& \leq & d_\theta (\underline{x}^2,\underline{x}^1) \left( L_{\Sigma} (f) ||\mu_0||_\infty 
		+ k |\mu_0|_\theta ^\omega \right).
	\end{eqnarray*}Thus, $$|\overline{f}|_\theta \leq L_{\Sigma} (f)||\mu_0||_\infty 
	+ k |\mu_0|_\theta ^\omega$$ and so $$|\overline{f}|_\theta \leq L_{\Sigma}(f)||\mu_0||_\infty 
	+ k |\mu_0|_\theta .$$ Therefore, $\overline{f} \in \mathcal{F}_\theta (\Sigma^+_A)$.
	
	It remains to show that $\nu \in \mathcal{L}^\infty$. Indeed, since $\nu|_\gamma = \func {\pi _{2,\gamma}^*}  \nu _\gamma$ and by equations (\ref{fjgh}) and (\ref{gdfgdgf}), we have

	\begin{eqnarray*}
		\left |\int_K{g}d(\nu |_\gamma) \right |&=&\left|\int_K{g\circ \pi _{2,\gamma}}d(\nu _\gamma)\right|\\&\leq & ||g||_\infty \dfrac{\int_K{f(\underline{x},y)}d\mu_{0, \gamma}(y)}{\overline{f}(\underline{x})} 
		\\&\leq &1, \ \mbox{where} \ \gamma=\gamma_{\underline{x}}. 
	\end{eqnarray*}Hence, $\nu \in S^\infty$ and $f \in \Theta _{\mu _0}$.

\end{proof}

\begin{proof}(of Theorem \ref{skdsdas})

	It is a direct consequence of propositions \ref{kkkskdjd} and \ref{ksjdhsdd}. 
\end{proof}

\subsection{Decay of Correlations for Random Dynamical Systems (IFS)}
Consider an IFS (of Example \ref{dereroidf}) given by the finite family of contractions $\mathcal{C}=\{\varphi_1, \cdots, \varphi_N\}$, where $\varphi _i:K \longrightarrow K$, $i=1, \cdots, N$ and $\mu$ its Hutchinson's invariant measure. We define for $\underline{x} \in \Sigma_A^+$ the map  $\Phi_n(\underline{x}):K\rightarrow K$ by $$\Phi_n(\underline{x})(y)=\varphi_{x_n}\circ \varphi_{x_{n-1}} \circ \cdots \circ\varphi_{x_0}(y).$$ 

\begin{definition}
	Let $\mathcal{B}_1$ and $\mathcal{B}_2$ be normed spaces of real valued functions $K \longmapsto \mathbb{R}$ with norms $|| \cdot ||_{\mathcal{B}_1}$ and $|| \cdot ||_{\mathcal{B}_2}$, respectively. Define the $\underline{x}$-\textit{coefficient of correlation between $g_1\in \mathcal{B}_1$ and $g_2\in \mathcal{B}_2$}, $C_n(g_1,g_2)(\underline{x})$, by 
	\begin{equation}
		C_n(g_1,g_2)(\underline{x}):=  \int_K{g_1   (g_2\circ \Phi_n(\underline{x}))}d\mu - \int_K{g_1  }d\mu \int_K{g_2 \circ \Phi_n(\underline{x})  }d\mu. 
	\end{equation}We say that the IFS, $\mathcal{C}=\{\varphi_1, \varphi_2, \cdots, \varphi_N\}$ has \textit{exponential decay of correlations} over $\mathcal{B}_1$ and $\mathcal{B}_2$ if there exists constants $0<\xi<1$ and $M>0$, such that 
	\begin{equation}
		\left \vert\int_{\Sigma^{+}_A} C_n(g_1,g_2)(\underline{x}) dm(\underline{x}) \right \vert \leq M\xi ^{n}||g_1 ||_{\mathcal{B}_1} ||g_2 ||_{\mathcal{B}_2}, \ \ \forall n \geq 1,
	\end{equation} for all $g_1 \in \mathcal{B}_1$ and $g_2 \in \mathcal{B}_2$.
\end{definition}

\begin{theorem}
	Let $\mathcal{C}=\{\varphi_1, \varphi_2, \cdots, \varphi_N\}$, a hyperbolic IFS and $\mu$ its invariant measure. Then, $\mathcal{C}$ has exponential decay of correlations over the real Lipschitz functions, $\widehat{\mathcal{F}}(K)$. 
\end{theorem}
\begin{proof}
	Let $g_1, g_2:K\longrightarrow \mathbb{R}$ be Lipschitz observables and suppose that $\overline{g}_1$ and $\overline{g}_2$ are their trivial extensions to $\Sigma$, $\overline{g}_i: = g_i \circ \pi _2$, $i=1,2$. They are still Lipschitz functions, but defined on the set $\Sigma$.  
	\begin{eqnarray*}
		C_n(g_1,g_2)(\underline{x})&=& \int_K{g_1(y) (g_2\circ \Phi _n(\underline{x}))(y)}d \mu(y)  \\&-& \int_K{g_1(y)} d\mu(y)\int_K{ (g_2\circ \Phi _n(\underline{x}))(y)}d \mu(y)  \\&=&  \int_K{g_1(y) g_2\circ \pi _2 \circ F^n (\underline{x}, y)}d \mu(y) \\&-& \int_K{g_1(y)} d\mu(y)\int_K{ g_2\circ \pi _2 \circ F^n (\underline{x}, y)}d \mu(y)\\&=& \int_K{g_1(y) \overline{g}_2\circ F^n (\underline{x}, y)}d \mu(y) \\&-& \int_K{g_1(y)} d\mu(y)\int_K{ \overline{g}_2\circ F^n (\underline{x}, y)}d \mu(y).
	\end{eqnarray*}Integrating with respect to the Markov's measure $m$ and applying Fubinis's Theorem, over $\Sigma_A^+$ we have
	\begin{eqnarray*}
		\int _{\Sigma_A^+} C_n(g_1,g_2)(\underline{x}) d m &=&\int _{\Sigma _A^+} \int_K{g_1 \overline{g}_2\circ F^n (\underline{x}, y)}d \mu(y)dm \\&-& \int _{\Sigma _A^+}\int_K{g_1(y)} d\mu(y)dm \int _{\Sigma _A^+}\int _K{ \overline{g}_2\circ F^n (\underline{x}, y)}d \mu(y)dm \\&=&\int _{\Sigma _A^+ \times K}{\overline{g}_1 \overline{g}_2\circ F^n (\underline{x}, y)}d(m \times \mu) \\&-& \int _{\Sigma _A^+ \times K}{\overline{g}_1} d(m \times \mu) \int _{\Sigma _A^+ \times K}{ \overline{g}_2\circ F^n (\underline{x}, y)}d(m \times \mu)\\&=&\int _{\Sigma}{\overline{g}_1 \overline{g}_2\circ F^n (\underline{x}, y)}d\mu _0 \\&-& \int _{\Sigma}{\overline{g}_1} d\mu _0 \int _{\Sigma}{ \overline{g}_2\circ F^n (\underline{x}, y)}d\mu _0.
	\end{eqnarray*}Where the last equality holds by the equation (\ref{oyiuotyu}) of Example \ref{dereroidf} (application of Theorem \ref{probun}).
	
	Then, $$\int _{\Sigma_A^+} C_n(g_1,g_2)(\underline{x}) d m = \int _{\Sigma}{\overline{g}_1 \overline{g}_2\circ F^n (\underline{x}, y)}d\mu _0 - \int _{\Sigma} {\overline{g}_1} d\mu _0 \int _{\Sigma} { \overline{g}_2\circ F^n (\underline{x}, y)}d\mu _0$$and by Theorem \ref{skdsdas}, the proof is complete.

\end{proof}

\section{Appendix}

\subsection{On Disintegration of Measures}\label{disint}
In this section, we prove some results on disintegration of absolutely continuous measures with respect to a measure $\mu_0 \in \mathcal{AB}$. Specifically, we are going to prove the Lemma \ref{hdgfghddsfg}. 

Let us fix some notation. Denote by $(\Sigma _A ^+,m)$ and $K$ the spaces defined in Section \ref{sec2}. For a $\mu_0$-integrable function $f: \Sigma _A ^+ \times K \longrightarrow \mathbb{R}$ and a pair $(\underline{x},y) \in \Sigma _A ^+ \times K$ ($\underline{x} \in \Sigma _A ^+$ and $y \in K$) we denote by $f_\gamma : K \longrightarrow \mathbb{R}$, the function defined by $f_\gamma(y) = f(\underline{x},y)$ and $f|_\gamma$ the restriction of $f$ on the set $\gamma = \gamma _{\underline{x}}:=\{\underline{x}\} \times K$. Then $f_\gamma = f|_\gamma \circ \pi _{2,\gamma}^{-1}$ and $f_\gamma \circ \pi_{2,\gamma} = f|_\gamma$, where $\pi _{2,\gamma}$ is the restriction of the projection $\pi _2(\underline{x},y):=y$ on the set $\gamma = \gamma_{\underline{x}}$. When no confusion is possible, we will denote the leaf $\gamma _{\underline{x}}$, just by $\gamma$.

From now on, for a given positive measure $\mu \in \mathcal{AB}$, on $\Sigma _A ^+ \times K$, $\widehat{\mu}$ stands for the measure $\pi_1^{*} \mu$, where $\pi_1$ is the projection on the first coordinate, $\pi_1(x,y)=x$.

For each measurable set $E \subset \Sigma _A ^+$, define $g: \Sigma _A ^+ \longrightarrow \mathbb{R}$ by $$g(\underline{x})= \phi_1(\underline{x})\int{\chi _{ \pi _1 ^{-1}(A)}|_\gamma (y) f|_\gamma(y)}d\mu_{0,\gamma}(y) \ \ \textnormal{where} \ \gamma = \gamma _{\underline{x}},$$ and note that

\begin{equation*}
	g(\underline{x})=
	\begin{cases}
		\phi_1(\underline{x}) \displaystyle{\int {f|_\gamma (y)}d \mu_{0,\gamma}}, \ \hbox{if} \ \underline{x} \in E \\
		0, \ \hbox{if} \ \underline{x} \notin E. %VERIFICAR
	\end{cases}
\end{equation*}Then, it holds $$g(\underline{x}) = \chi _E (\underline{x}) \displaystyle{\phi_1(\underline{x})  \int {f|_\gamma (y)}d \mu_{0,\gamma}},$$where $\gamma = \gamma _{\underline{x}}$.

\begin{proof}{(of Lemma \ref{hdgfghddsfg})}
	
	For each measurable set $E \subset \Sigma _A ^+$, we have
	\begin{eqnarray*}
		\int_E{\dfrac{\pi _1 ^* (f\mu_0)}{dm}}dm &=&\int{\chi _E \circ \pi _1}d(f\mu_0)
		\\&=&\int{\chi _{ \pi _1 ^{-1}(E)} f}d\mu_0
		\\&=&\int \left[\int{\chi _{ \pi _1 ^{-1}(E)}|_\gamma (y) f|_\gamma(y)}d\mu_{0,\gamma}(y)\right]d(\phi_1 m)(\underline{x})
		\\&=&\int \left[\phi_1(\underline{x})\int{\chi _{ \pi _1 ^{-1}(E)}|_\gamma (y) f|_\gamma(y)}d\mu_{0,\gamma}(y)\right]dm(\underline{x})
		\\&=&\int{g(\underline{x})}dm(\underline{x})
		\\&=&\int_E {\left[\int{f_\gamma(y)}d\mu_{0}{|_\gamma}(y)\right]}dm(\underline{x}).
	\end{eqnarray*}Thus, it holds 
	
	\begin{equation*}
		\dfrac{\pi _1 ^* (f\mu_0)}{dm} (\underline{x}) =  \int{f_\gamma(y)}d\mu_{0}{|_\gamma}, \ \hbox{for} \ m-\hbox{a.e.} \ \underline{x} \in \Sigma _A ^+.
	\end{equation*}And by a straightforward computation 
	
	\begin{equation}\label{gh}
		\dfrac{\pi _1 ^* (f\mu_0)}{dm}(\underline{x}) = \phi_1 (\underline{x}) \int{f|_\gamma(y)}d\mu_{0,\gamma}, \ \hbox{for} \ m-\hbox{a.e.} \ \underline{x} \in \Sigma _A ^+.
	\end{equation}Thus, equation (\ref{fjgh}) is established.
	
	\begin{remark}\label{ghj}
		Setting
		\begin{equation}\label{tyu}
			\overline{f}:= \dfrac{\pi _1 ^* (f\mu_0)}{dm},
		\end{equation}we get, by equation (\ref{gh}), $\overline{f}(\underline{x})=0$ iff $\phi_1 (\underline{x}) = 0$ or $\displaystyle{\int{f|_\gamma(y)}d\mu_{0,\gamma} (y)=0}$, for  $m$-a.e. $\underline{x} \in \Sigma _A ^+$. 
	\end{remark}

	Now, let us see that, by the $\widehat{\nu}$-uniqueness of the disintegration, equation (\ref{gdfgdgf}) holds. To do it, define, for $m$-a.e. $\underline{x} \in \Sigma _A ^+$ (where $\gamma = \gamma _{\underline{x}}$), the function $h_\gamma : K \longrightarrow \mathbb{R}$, in a way that
	
	\begin{equation}\label{jri}
		h_\gamma (y) =
		\begin{cases}
			\dfrac{f|_\gamma (y)}{\int{f|_\gamma(y)}d\mu_{0,\gamma}(y)} , \ \hbox{if} \ \underline{x} \in B ^c \\
			0, \ \hbox{if} \ \underline{x} \in B,
		\end{cases}
	\end{equation}where $B :=  \overline{h} ^{-1}(0)$ and $\gamma=\gamma_{\underline{x}}=\{ \underline{x} \}\times K$. Let us prove equation (\ref{gdfgdgf}) by showing that, for all measurable set $E \subset \Sigma _A ^+ \times K$, it holds $$f \mu _0 (E)  = \int _{\Sigma _A ^+} {\int _{E \cap \gamma} {h_\gamma(y)}}d\mu _{0, \gamma} (y)d (\pi_1 {_*}(f \mu_0))(\underline{x}).$$In fact, by equations (\ref{gh}), (\ref{tyu}), (\ref{jri}) and Remark \ref{ghj}, we get
	
	\begin{eqnarray*}
		f\mu_0 (E) &=& \int _E {f} d\mu_0
		\\&=& \int _{\Sigma _A ^+} \int _{E\cap \gamma} {f|_\gamma} d\mu_{0, \gamma}d (\phi_1 m)(\underline{x})
		\\&=& \int _{B^c} \int _{E\cap \gamma} {f|_\gamma} d\mu_{0, \gamma}d (\phi_1 m)(\underline{x})
		\\&=& \int _{B^c} \int{f|_\gamma(y)}d\mu_{0,\gamma}(y)\phi_1 (\underline{x})\left[ \dfrac{1}{\int{f|_\gamma(y)}d\mu_{0,\gamma}(y)}\int _{E\cap \gamma} {f|_\gamma}  d\mu_{0, \gamma}\right]d m(\underline{x})
		\\&=& \int _{B^c} \overline{f}(\underline{x})\left[ \dfrac{1}{\int{f|_\gamma(y)}d\mu_{0,\gamma}(y)}\int _{E\cap \gamma} {f|_\gamma}  d\mu_{0, \gamma}\right]d m(\underline{x})
		\\&=& \int _{B^c} \left[ \dfrac{1}{\int{f|_\gamma(y)}d\mu_{0,\gamma}(y)}\int _{E\cap \gamma} {f|_\gamma}  d\mu_{0, \gamma}\right]d \overline{f}m(\underline{x})
		\\&=& \int _{B^c} {\int _{E \cap \gamma} {h_\gamma(y)}}d\mu _{0, \gamma} (y)d (\pi_1 {_*}(f \mu_0))(\underline{x})
		\\&=& \int _{\Sigma _A ^+} {\int _{E \cap \gamma} {h_\gamma(y)}}d\mu _{0, \gamma} (y)d (\pi_1 {_*}(f \mu_0))(\underline{x}).
	\end{eqnarray*}And we are done.
	
\end{proof}

\subsection{Lifting Invariant Measures}\label{lifting}

In this section, due to the general nature of the construction, we will use broader notations for our dynamical system. Specifically, we consider a system $F: M_1 \times M_2 \longrightarrow M_1 \times M_2$, defined by $F(x,y) = (T(x), G(x,y))$, where $T: M_1 \longrightarrow M_1$ and $G:M_1 \times M_2 \longrightarrow  M_2$ are measurable transformations satisfying certain assumptions described below. In particular, we assume that $T: M_1 \longrightarrow M_1$ admits an invariant measure $\mu_1$ on $M_1$. We show that $F$ admits an invariant measure $\mu_0$, such that $\pi _1^*\mu_0 = \mu_1$, where $\pi _1^*\mu_0 (E) := \mu_0 (\pi ^{-1}(E))$. This result extends the one presented in Subsection 7.3.4 (page 225) of \cite{AP}.
\subsubsection{Lifting Measures}

Let $(M_1, d_1)$ and $(M_2, d_2)$ be two compact metric spaces and suppose that they are endowed with their Borelean sigma algebras $\mathcal{B}_1$ and $\mathcal{B}_2$, respectively. Moreover, let $\mu_1$ and $\mu_2$ be measures on $\mathcal{B}_1$ and $\mathcal{B}_2$, respectively, where $\mu_1$ is a $T$-invariant probability for the measurable transformation $T:M_1 \longrightarrow M_1$. Also denotes by $\pi_1: M_1 \times M_2 \longrightarrow M_1$ the projection on $M_1$, $\pi_1(x,y)=x$ for all $x\in M_1$ and all $y \in M_2.$ 

Consider a measurable map $F : M_1 \times M_2 \longrightarrow M_1 \times M_2$,  defined by $F(x,y)=(T(x),G(x,y))$, where $G(x,y): M_1 \times M_2 \longrightarrow M_2$ satisfies:

\begin{enumerate}
	\item [$\bullet$] There exists a measurable set $A_1 \subset M_1$, with $\mu_1(A^c)=0$, such that for each  $x \in A_1$, it holds $T^n(x) \in A_1$ for all $n \geq 0$ and 
	\begin{equation}
		\label{1}
		\lim _{n\longrightarrow +\infty } \diam F^n(\gamma_x) = 0
	\end{equation} uniformly on $A_1$, where $\gamma_x := \{x\}\times M_2$ and the diameter of the set $F^n(\gamma_x)$ is defined with respect to the metric $d = d_1 + d_2$. 
\end{enumerate}For this sort of map, it holds $\displaystyle{F(\gamma_x) \subset \gamma_{T(x)}}$, for all $x$.

Let $C^0(M_1 \times M_2)$ be the space of the real continuous functions, $\psi: M_1 \times M_2 \longrightarrow \mathbb{R}$, endowed with the $||\cdot ||_\infty$ norm. 

For each $\psi \in C^0(M_1 \times M_2) $ define $\psi _+:M_1 \longrightarrow \mathbb{R}$ and $\psi _-:M_1 \longrightarrow \mathbb{R}$ by

\begin{equation}
	\psi _+ (x)= \sup _{\gamma _x} \psi(x,y),
\end{equation}and

\begin{equation}
	\psi _- (x)= \inf _{\gamma _x}\psi(x,y).
\end{equation}

\begin{lemma}
	Both limits $$\lim {\int{(\psi \circ F^n)_+}d\mu_1}$$ and $$\lim {\int{(\psi \circ F^n)_-}d\mu_1}$$ do exist and they are equal.
\end{lemma}

\begin{proof}
	
	Since $(M_1 \times M_2, d)$ is a compact metric space and $\psi  \in C^0(M_1 \times M_2) $, $\psi$ is uniformly continuous.
	
	Given $\epsilon > 0$, let $\delta >0$ be such that $|\psi(x_1,y_1) - \psi(x_2,y_2)|<\epsilon$ if $d((x_1,y_1),(x_2,y_2))<\delta$ and $n_0 \in \mathbb{N}$ such that $$n> n_0 \Rightarrow \diam F^n(\gamma_x) < \delta,$$ for all $x \in A_1$.
	
	For all $k \in \mathbb{N}$, it holds 
	
	\begin{equation}
		\label{pos}
		(\psi \circ F^{n+k})_- (x)- (\psi \circ F^{n})_- (T^k(x)) = \inf_{F^k (\gamma_x)} \psi \circ F^{n} - \inf_ {\gamma_{T^k(x)}}{\psi \circ F^{n}} \geq 0. 
	\end{equation}
	Moreover, for all $x \in A_1$

	\begin{eqnarray*}
		(\psi \circ F^{n+k})_- (x)- (\psi \circ F^{n})_- (T^k(x))
		&=&\inf_{F^k (\gamma_x)} \psi \circ F^{n} - \inf_ {\gamma_{T^k(x)}}{\psi \circ F^{n}}
		\\&\leq&\sup_{F^k (\gamma_x)} \psi \circ F^{n} - \inf_ {\gamma_{T^k(x)}}{\psi \circ F^{n}}
		\\&\leq&\sup_{\gamma_{T^k(x)}} \psi \circ F^{n} - \inf_ {\gamma_{T^k(x)}}{\psi \circ F^{n}}
		\\&\leq&\sup_{F^{n}(\gamma_{T^k(x)})} \psi   - \inf_ {F^{n}(\gamma_{T^k(x)})}{\psi}
		\\&\leq& \epsilon.
	\end{eqnarray*}Thus, $$(\psi \circ F^{n+k})_- (x)- (\psi \circ F^{n})_- (T^k(x))\leq \epsilon \ \forall \ x \in A_1.$$Integrating over $M_1$, since $\mu_1(A_1^c)=0$, we get $$\int{(\psi \circ F^{n+k})_-} d\mu_1 (x)- \int{(\psi \circ F^{n})_-} \circ T^kd\mu_1 \leq \epsilon.$$Since $\mu_1$ is $T$-invariant and by (\ref{pos}) we get $$\left|\int{(\psi \circ F^{n+k})_-} d\mu_1 (x)- \int{(\psi \circ F^{n})_-}d\mu_1\right| \leq \epsilon \ \ \forall \ \ n \geq n_0 \  \textnormal{and} \ \forall \  k \in \mathbb{N},$$which shows that the sequence $(\int{(\psi \circ F^{n})_-}d\mu_1)_n$ converges in $\mathbb{R}$. Let us denote 
	
	\begin{equation}
		\mu(\psi):= \lim_{n\longrightarrow +\infty } {\int{(\psi \circ F^{n})_-}d\mu_1}.
	\end{equation}

	It is remaining to show that, it also holds $\displaystyle{\mu(\psi)= \lim_{n\longrightarrow +\infty } {\int{(\psi \circ F^{n})_+}d\mu_1}}$. Indeed, given $\epsilon > 0$, let $\delta >0$ be such that $|\psi(x_1,y_1) - \psi(x_2,y_2)|<\dfrac{\epsilon}{2} $ if $d((x_1,y_1),(x_2,y_2))<\delta$.
	
	Consider $n_0 \in \mathbb{N}$ such that 
	
	\begin{equation}
		n> n_0 \Rightarrow \diam F^n(\gamma_x) < \delta,
	\end{equation} for all $x \in A_1$ and 
	
	\begin{equation}\label{tdf}
		n> n_0 \Rightarrow \left|\mu(\psi)- \int{(\psi \circ F^{n})_-}d\mu_1\right| < \dfrac{\epsilon}{2}.
	\end{equation} Then,
	
	\begin{eqnarray*}
		0 
		&\leq& (\psi \circ F^{n})_+ (x)- (\psi \circ F^{n})_- (x)
		\\&=& \sup_{F^{n}(\gamma_{x})} \psi   - \inf_ {F^{n}(\gamma_{x})}{\psi}
		\\&\leq& \dfrac{\epsilon}{2}.
	\end{eqnarray*}The above inequality yields, for all $n>n_0$
	
	\begin{equation}\label{djfsdf}
		\left|\int{(\psi \circ F^{n})_+} d\mu_1 (x)- \int{(\psi \circ F^{n})_-}d\mu_1\right|\leq \dfrac{\epsilon}{2}. 
	\end{equation}
	
	Therefore, if $n>n_0$, by (\ref{tdf}) and (\ref{djfsdf}) it holds
	
	\begin{eqnarray*}
		\left|\mu(\psi)- \int{(\psi \circ F^{n})_+}d\mu_1\right|
		&\leq& \left|\mu(\psi)- \int{(\psi \circ F^{n})_-}d\mu_1\right| 
		\\&+& \left|\int{(\psi \circ F^{n})_+}d\mu_1v- \int{(\psi \circ F^{n})_-}d\mu_1\right|
		\\&<&\epsilon.
	\end{eqnarray*}And the proof is complete.
	
\end{proof}

\begin{lemma}
	The function $\mu: C^0(M_1 \times M_2) \rightarrow \mathbb{R}$ defined for each $\psi \in C^0(M_1 \times M_2)$ by $\displaystyle{\mu(\psi):= \lim_{n\rightarrow +\infty } {\int{(\psi \circ F^{n})_-}d\mu_1}=\lim_{n\rightarrow +\infty } {\int{(\psi \circ F^{n})_+}d\mu_1}}$, defines a bounded linear functional such that $\mu(1)=1$ and $\mu(\psi)\geq 0$ for all $\psi \geq 0$.
\end{lemma}

\begin{proof}
	Is straightforward to show that $\mu(1)=1$, $\mu(\psi)\geq 0$ for all $\psi \geq 0$ and $\mu(\psi) \leq 1$ for all $\psi$ in the unit ball, $||\psi||_\infty \leq 1$. Then $\mu$ is bounded. Let us see, that $\mu(\alpha\psi_1 + \psi_2) = \alpha\mu(\psi_1) + \mu(\psi_2)$, for all $\psi_1, \psi_2 \in C^0(M_1 \times M_2)$ and all $\alpha \in \mathbb{R}$.
	
	Define $A_\psi^x :=\{\psi(x,y); (x,y)\in \gamma_x\}$ and note that $\psi_+(x) = \sup A_\psi^x$ and $\psi_-(x) = \inf A_\psi^x$. Moreover, $A_{\psi_1 + \psi_2}^x \subset A_{\psi_1}^x + A_{ \psi_2}^x$. This implies 
	
	\begin{equation*}
		\sup A_{\psi_1 + \psi_2}^x \leq \sup A_{\psi_1}^x + \sup A_{ \psi_2}^x
	\end{equation*}and
	
	\begin{equation*}
		\inf A_{\psi_1 + \psi_2}^x \geq \inf A_{\psi_1}^x + \inf A_{ \psi_2}^x
	\end{equation*}which gives $(\psi_1 + \psi_2)_+ \leq (\psi_1)_+ + (\psi_2)_+$ and $(\psi_1 + \psi_2)_- \geq (\psi_1)_- + (\psi_2)_-$. Then,

	\begin{equation*}
		((\psi_1 + \psi_2) \circ F^n)_+ \leq (\psi_1 \circ F^n)_+ + (\psi_2\circ F^n)_+
	\end{equation*}and
	
	\begin{equation*}
		((\psi_1 + \psi_2) \circ F^n)_- \geq (\psi_1 \circ F^n)_- + (\psi_2\circ F^n)_-.
	\end{equation*}Integrating and taking the limit we have

	\begin{equation*}
		\lim_{n\rightarrow +\infty } \int{((\psi_1 + \psi_2) \circ F^n)_+}d\mu_1 \leq \lim_{n\rightarrow +\infty } \int{(\psi_1 \circ F^n)_+}d\mu_1 + \lim_{n\rightarrow +\infty } \int{(\psi_2\circ F^n)_+}d\mu_1
	\end{equation*}and
	
	\begin{equation*}
		\lim_{n\rightarrow +\infty } \int{((\psi_1 + \psi_2) \circ F^n)_-}d\mu_1 \geq \lim_{n\rightarrow +\infty } \int{(\psi_1 \circ F^n)_-}d\mu_1 + \lim_{n\rightarrow +\infty } \int{(\psi_2\circ F^n)_-}d\mu_1.
	\end{equation*}This implies $\mu(\psi_1+\psi_2) = \mu(\psi_1)+\mu(\psi_2)$.

	For $\alpha \geq 0$, we have 	
	\begin{eqnarray*}
		\mu(\alpha \psi)
		&=&\lim_{n\rightarrow +\infty } \int{((\alpha\psi) \circ F^n)_+}d\mu_1 
		\\&=&\alpha\lim_{n\rightarrow +\infty } \int{(\psi \circ F^n)_+}d\mu_1
		\\&=&\alpha \mu(\psi).
	\end{eqnarray*}
	For $\alpha <0$, it holds 
	\begin{eqnarray*}
		\mu(\alpha \psi)
		&=&\lim_{n\rightarrow +\infty } \int{((\alpha\psi) \circ F^n)_+}d\mu_1 
		\\&=&\alpha\lim_{n\rightarrow +\infty } \int{(\psi \circ F^n)_-}d\mu_1
		\\&=&\alpha \mu(\psi).
	\end{eqnarray*}And the proof is complete.
	
\end{proof}	

As a consequence of the above results and the Riez-Markov Lemma we get the following theorem.

\begin{theorem}\label{kjdhkskjfkjskdjf}
	There exists a unique measure $\mu_0$ on $M_1 \times M_2$ such that for every continuous function $\psi \in C^0 (M_1 \times M_2)$ it holds

	\begin{equation}
		\mu(\psi)=\int {\psi}d\mu_0. 
	\end{equation}
\end{theorem}

\begin{proposition}\label{kjdhkskjfkjskdjff}
	The measure $\mu_0$ is $F$-invariant and $\pi_1^{\ast}\mu_0 = \mu_1$.
\end{proposition}

\begin{proof}
	Denote by $\func{F}^{\ast}\mu _0$ the measure defined by $\func{F}^{\ast}\mu _0(E) = \mu_0(F^{-1}(E))$ for all measurable set $E \subset M_1 \times M_2$. For each $\psi \in C^0 (M_1 \times M_2)$ we have
	\begin{eqnarray*}
		\int {\psi}d\func{F}^{\ast}\mu _0
		&=&\int {\psi \circ F}d\mu _0 
		\\&=& \mu(\psi)
		\\&=& \lim_{n\rightarrow +\infty } \int{(\psi \circ F^n \circ F)_+}d\mu_1
		\\&=& \mu(\psi)	
		\\&=& \int {\psi }d\mu _0 .
	\end{eqnarray*}	It implies that $\func{F}^{\ast}\mu _0 = \mu_0$ and we are done.

	To prove that $\pi_1^{\ast}\mu_0 = \mu_1$, consider a continuous function $\phi: M_1 \longrightarrow \mathbb{R}$. Then, $\phi \circ \pi _1 \in C^0 (M_1 \times M_2)$. Therefore,
	\begin{eqnarray*}
		\int{\phi}d\pi_1^{\ast}\mu_0 
		&=& \int{\phi \circ \pi _1}d\mu_0
		\\&=& \lim_{n\rightarrow +\infty } \int{(\phi \circ \pi _1 \circ F^n )_+}d\mu_1 
		\\&=& \lim_{n\rightarrow +\infty } \int{(\phi \circ T^n)_+}d\mu_1 
		\\&=& \lim_{n\rightarrow +\infty } \int{\phi \circ T^n}d\mu_1 
		\\&=& \lim_{n\rightarrow +\infty } \int{\phi}d\mu_1
		\\&=&  \int{\phi}d\mu_1.
	\end{eqnarray*}Thus, $\pi_1^{\ast}\mu_0 = \mu_1$.
	
\end{proof}

\subsection{Linearity of the restriction \label{remmm}}

Consider the measurable spaces $\Sigma _A ^+$, $K$, and $\Sigma_A^+ \times K$ defined in Subsections \ref{sec2} and \ref{sec22}. Let $\mu \in \mathcal{AB}$ be a positive measure on the measurable space $(\Sigma, \mathcal{B})$, where $\Sigma = \Sigma_A^+ \times K$ and consider its disintegration $(\{\mu_{\gamma_{\underline{x}}}\} _{\gamma_{\underline{x}}}, \phi _1 m )$ along $\mathcal{F}^s$, where $\phi _ 1:= \dfrac{d\pi _1 ^{\ast}\mu}{dm} \in L^1_ m$. Henceforth, to simplify the notation, we denote $\gamma _{\underline{x}}$ by $\gamma$.

The proof of the following proposition is a direct consequence of Proposition 5.1.7, page 150 of \cite{Kva}, so we omit it. 
\begin{proposition}
	Suppose that $\mathcal{B}$ has a countable generator, $\Gamma$. If $\{\mu_\gamma\}_\gamma$ and $\{\mu'_\gamma\}_\gamma$ are disintegrations of a positive measure $\mu$ along $\mathcal{F}^s$, then $$\phi_1(\underline{x})\mu_{\gamma_{\underline{x}}} = \phi_1(\underline{x})\mu'_{\gamma_{\underline{x}}}$$
	$m$-a.e. $\underline{x} \in \Sigma^{+}_A.$
	\label{againa}
\end{proposition}

\begin{proposition}\label{hytre}
	Let $\mu_1,\mu_2 \in \mathcal{AB}$ be two positive measures and denote their marginal densities by $\dfrac{d{\mu}_1}{dm}= \phi_1 $ and $\dfrac{d\mu _2}{dm} = \psi_1 $ , where $\phi_1, \psi_1\in L^1_m$ respectively. Then $(\mu_1 + \mu_2)|_\gamma = \mu_1|_\gamma + \mu_2|_\gamma,$ where $\gamma=\gamma_{\underline{x}}$ for $m$-a.e. $\underline{x} \in \Sigma^+_A$.
\end{proposition}
\begin{proof}
	Note that $\dfrac{d(\mu_1 + \mu_2)}{dm}= \phi _1 + \psi_1$. Moreover, consider the disintegration of $\mu _1 + \mu_2$ given by $$(\{(\mu_1 + \mu_2) _\gamma\}_\gamma, (\phi _1 + \psi_1) m),$$ where
	
	\begin{equation*}
		(\mu_1 + \mu_2) _\gamma =
		\begin{cases}
			\dfrac{\phi _1(\underline{x})}{\phi _1(\underline{x}) + \psi _1(\underline{x})} \mu _{1,\gamma}+ \dfrac{\psi _1(\underline{x})}{\phi _1(\underline{x}) + \psi _1(\underline{x})} \mu _{2,\gamma}, \ \textnormal{if} \ \phi _1(\underline{x}) + \psi _1(\underline{x}) \neq 0  \\ 0,\ \textnormal{if} \ \phi _1(\underline{x}) + \psi _1(\underline{x}) = 0.
		\end{cases}
	\end{equation*}Then, by Proposition \ref{againa} for $m$-a.e. $\underline{x} \in \Sigma_A^+$, it holds $$(\phi _1 + \psi_1)(\underline{x})(\mu_1 + \mu_2) _\gamma = \phi _1(\underline{x}) \mu _{1,\gamma} + \psi _1(\underline{x}) \mu _{2,\gamma},$$ where $\{\mu_{1,\gamma}\}_{\gamma}$ and $\{\mu_{2,\gamma}\}_\gamma$ are the disintegrations of $\mu_1$ and $\mu_2$ respectively. Therefore, $(\mu_1 + \mu_2)|_\gamma = \mu_1|_\gamma + \mu_2|_\gamma$ $m$-a.e. $\underline{x} \in \Sigma_A^+$.
	
\end{proof}

\begin{definition}\label{disj}
	We say that a positive measure $\lambda_1$ is disjoint from a positive measure $\lambda _2$ if $(\lambda_1 - \lambda _2)^+ = \lambda _1$ and $(\lambda_1 - \lambda _2)^- = \lambda _2$.
\end{definition}

\begin{remark}
	A straightforward computations yields that if $\lambda _1 + \lambda _2$ is disjoint from $\lambda _3$, then both $\lambda _1$ and $\lambda _2$ are disjoint from $\lambda _3$, where $\lambda_1, \lambda_2$ and $\lambda_3$ are all positive measures. 
\end{remark}

\begin{lemma}Suppose that $\mu =\mu ^{+}-\mu ^{-}$ and $\nu =\nu ^{+}-\nu ^{-}$ are the Jordan decompositions of the signed measures $\mu$ and $\nu$. Then, there exist positive measures $\mu_1$, $\mu_2$, $\mu ^{++}$, $\mu ^{--}$, $\nu ^{++}$ and $\nu ^{--}$ such that $\mu ^{+}=\mu ^{++}+\mu_{1}$ $\mu ^{-}=\mu ^{--}+\mu_{2}$ and  $\nu ^{+}=\nu ^{++}+\mu_{2}$, $\nu ^{-}=\nu ^{--}+\mu_{1}$.
	\label{fhdj}
\end{lemma}
\begin{proof}
	Suppose $\mu = \nu_1 - \nu_2$ with $\nu_1$ and $\nu_2$ positive measures. Let $\mu^+ $ and $\mu^-$ be the Jordan decomposition of $\mu$.
	Let $\mu' = \nu_1 - \mu^+$, then $\nu_1 = \mu^- + \mu'$. Indeed $\mu^+ - \mu ^- = \nu _1 - \nu _2$ which implies that $\mu^+ - \nu_1 = \mu^- - \nu_2$.
	Thus if $\nu_1, \nu_2$ is a decomposition of $\mu$, then $\nu_1 = \mu^+ + \mu '$ and $\nu_2 = \mu^- + \mu '$ for some positive measure $\mu'$.
	Now, consider $\mu = \mu ^+ - \mu ^-$ and $\nu = \nu ^+ - \nu ^-$. Since the pairs of positive measures $\mu^+, \nu^-$ and $(\mu^+ - \nu^-)^+$, $(\mu ^+ - \nu ^-)^-$ are both decompositions of $\mu^+ - \nu ^-$, by the above comments, we get that $\mu^+ =(\mu^+ - \nu^-)^+ + \mu_1$ and $\nu^- =(\mu^+ - \nu^-)^- + \mu_1$, for some positive measure $\mu_1$. Analogously, since the pairs of positive measures $\mu^-, \nu^+$ and $(\nu^+ - \mu^-)^+$, $(\nu ^+ - \mu ^-)^-$ are both decompositions of $\nu^+ - \mu ^-$, by the above comments, we get that $\nu^+ =(\nu^+ - \mu^-)^+ + \mu_2$ and $\mu^- =(\nu^+ - \mu^-)^- + \mu_2$, for some positive measure $\mu_2$.
	By Definition \ref{disj}, $\mu+$ and $\mu^-$ are disjoint, and so are $(\mu ^+ - \nu^-)^+$ and $(\nu ^+ - \mu^-)^-$. Analogously, $\nu+$ and $\nu^-$ are disjoint, and so are $(\mu ^+ - \nu^-)^-$ and $(\nu ^+ - \mu^-)^+$. Moreover, since $(\mu ^+ - \nu^-)^+$ and $(\mu ^+ - \nu^-)^-$ are disjoint, so are $(\nu ^+ - \mu^-)^+$ and $(\nu ^+ - \mu^-)^-$. This gives that, the pair $(\mu ^+ - \nu^-)^+ + (\nu ^+ - \mu^-)^+$, $(\nu ^+ - \mu^-)^- + (\mu ^+ - \nu^-)^-$ is a Jordan decomposition of $\mu + \nu$ and we are done.
\end{proof}

\begin{proposition}
	Let $\mu,\nu \in \mathcal{AB}$ be two signed measures. Then $(\mu + \nu)|_\gamma = \mu|_\gamma + \nu|_\gamma$ $m_1$-a.e. $\gamma \in N_1$.
\end{proposition}
\begin{proof}
	
	Suppose that $\mu =\mu ^{+}-\mu ^{-}$ and $\nu =\nu ^{+}-\nu ^{-}$ are the Jordan decompositions of $\mu$ and $\nu$ respectively. By definition, $\mu |_{\gamma }=\mu ^{+}|_{\gamma }-\mu ^{-}|_{\gamma }$, $\nu |_{\gamma }=\nu ^{+}|_{\gamma }-\nu ^{-}|_{\gamma }$. 
	
	By Lemma \ref{fhdj}, suppose that $\mu ^{+}=\mu ^{++}+\mu_{1}$, $\mu ^{-}=\mu ^{--}+\mu_{2}$ and  $\nu^{+}=\nu ^{++}+\mu_{2}$, $\nu ^{-}=\nu ^{--}+\mu_{1}$. In a way that $(\mu +\nu )^{+}=\mu ^{++}+\nu ^{++}$ and $(\mu +\nu )^{-}=\mu ^{--}+\nu ^{--}$. By Proposition \ref{hytre}, it holds $\mu ^{+}|_{\gamma }=\mu ^{++}|_{\gamma }+\mu_{1}|_{\gamma }$, $\mu ^{-}|_{\gamma }=\mu ^{--}|_{\gamma }+\mu_{2}|_{\gamma }$, $\nu ^{+}|_{\gamma }=\nu ^{++}|_{\gamma }+\mu_{2}|_{\gamma }$ and $\nu ^{-}|_{\gamma }=\nu ^{--}|_{\gamma }+\mu_{1}|_{\gamma }$.
	
	Moreover,
	
	$(\mu +\nu )^{+}|_{\gamma }=\mu ^{++}|_{\gamma }+\nu ^{++}|_{\gamma }$
	
	$(\mu +\nu )^{-}|_{\gamma }=\mu ^{--}|_{\gamma }+\nu ^{--}|_{\gamma }$
	
	Putting all together, we get:
	
	\begin{eqnarray*}
		(\mu +\nu )|_{\gamma } &=&(\mu +\nu )^{+}|_{\gamma }-(\mu +\nu
		)^{-}|_{\gamma } \\
		&=&\mu ^{++}|_{\gamma }+\nu ^{++}|_{\gamma }-(\mu ^{--}|_{\gamma }+\nu
		^{--}|_{\gamma }) \\
		&=&\mu ^{++}|_{\gamma }+\mu_{1}|_{\gamma }+\nu ^{++}|_{\gamma }+\mu_{2}|_{\gamma
		}-(\mu ^{--}|_{\gamma }+\mu_{2}|_{\gamma }+\nu ^{--}|_{\gamma }+\mu_{1}|_{\gamma
		}) \\
		&=&\mu ^{+}|_{\gamma }-\mu ^{-}|_{\gamma }+\nu ^{+}|_{\gamma }-\nu
		^{-}|_{\gamma } \\
		&=&\mu |_{\gamma }+\nu |_{\gamma }.
	\end{eqnarray*}
\end{proof}We immediately arrive at the following
\begin{proposition}\label{lasttttt}
	Let $\mu \in \mathcal{AB}$ be a signed measure and $\mu = \mu^+ - \mu ^-$ its Jordan decomposition. If $\mu_1$ and $\mu_2$ are positive measures such that $\mu = \mu_1 - \mu _2$, then $ \mu|_\gamma =\mu_1|_\gamma - \mu_2|_\gamma$. It means that, the restriction does not depends on the decomposition of $\mu$.
\end{proposition}

\subsection*{Acknowledgment}
The authors would like to express their gratitude for the valuable comments provided by the reviewers, whose insightful suggestions and remarks significantly contributed to the improvement of the paper.
This work was partially supported by CNPq (Brazil) Grant Alagoas Din\^amica: 409198/2021-8 and 
\begin{enumerate}
	\item[(1)] FAPEAL (Alagoas-Brazil) Grant E:60030.0000002330/2022. 
	\item[(2)]FAPEAL (Alagoas-Brazil)Grant E:60030.0000000161/2022. 
\end{enumerate}

%	No que segue, dados dois espaços de Banach $(X,\|\cdot\|_1)$ e $(Y,\|\cdot\|_2)$ e operadores lineares $S: X\to X$ e $T: X\to Y$, denotamos por $\|S\|_1$ e $\|T\|_{2,1}$ as normas de $S$ e $T$. Isto é, são satisfeitas as relações $\|Sx\|_1\leq\|S\|_1\|x\|_1$ e $\|Tx\|_2\leq\|T\|_{2,1}\|x\|_1$ para todo $x\in X$.

%\begin{thebibliography}{1}
%\bibitem{test} A. B. C. Test, \textit{On a Test.} J. of Testing
%\textbf{88} (2000), 100--120.
%\bibitem{latex} G. Gr\"atzer, \textit{Math into \LaTeX.} 3rd Edition,
%Birkh\"auser, 2000.
%\end{thebibliography}

% ------------------------------------------------------------------------
\end{document}